\documentclass[12pt,a4paper]{article}
\usepackage{amsmath,enumerate,bm,bbm,color}
\usepackage{amsthm}
\usepackage{amssymb}
\usepackage{amsbsy}
\usepackage{amsfonts}
\usepackage{amstext}
\usepackage{amscd}
\usepackage{graphicx}
\usepackage{geometry}
\usepackage{hangcaption}
\usepackage{url}

\numberwithin{equation}{section}
\theoremstyle{plain}
\newtheorem{thm}{Theorem}[section]

\newtheorem{cor}[thm]{Corollary}
\newtheorem{lem}[thm]{Lemma}
\theoremstyle{definition}

\newtheorem{conj}[thm]{Conjecture}
\newtheorem{rem}[thm]{Remark}
\newtheorem{defi}[thm]{Definition}

\newcommand{\R}{\mathbb{R}}  % real numbers
\newcommand{\C}{\mathbb{C}}  % complex numbers 
\newcommand{\N}{\mathbb{N}}  % natural numbers
\newcommand{\im}{\text{\normalfont Im}} % imaginary part
\newcommand{\re}{\text{\normalfont Re}} % real part
\newcommand{\sign}{{\rm sign}} % signature 
\newcommand{\I}{{\rm i}} % a solution to z^2= -1
\renewcommand{\epsilon}{\varepsilon} 
\renewcommand{\ell}{l}

\newcommand{\PM}{\mathcal{P}}
\newcommand{\HCM}{{\rm HCM}}
\newcommand{\ME}{{\rm ME}}
\newcommand{\GGC}{{\rm GGC}}
\newcommand{\FID}{{\rm FID}}
\newcommand{\UI}{{\rm UI}}
\newcommand{\UIS}{{\rm UI}_s}
\newcommand{\ID}{{\rm ID}}
\newcommand{\FR}{{\rm FR}}
\newcommand{\B}{\mathfrak{B}}

\newcommand{\Beta}{{\bm\beta}} % beta distribution
\newcommand{\WS}{\mathcal{S}} % semicircle distribution
\newcommand{\MP}{{\normalfont \textbf{fp}}} % free Poisson distribution
\newcommand{\g}{{\bm\gamma}} % gamma distribution 
\newcommand{\bs}{{\rm\bf b }} % Boolean stable distribution

\newcommand{\be}{\begin{equation}}
\newcommand{\ee}{\end{equation}}
\newcommand{\wf}{w}  % density function

\begin{document}
\title{Free infinite divisibility for powers of random variables}

\author{Takahiro Hasebe} 

\date{\today}

\maketitle

\pagestyle{myheadings}\markboth{}{Free infinite divisibility for powers of random variables}

\begin{abstract}
We prove that $X^r$ follows an FID distribution if: (1) $X$ follows a free Poisson distribution without an atom at 0 and $r\in(-\infty,0]\cup[1,\infty)$; (2) $X$ follows a free Poisson distribution with an atom at 0 and $r\geq1$; (3) $X$ follows a mixture of some HCM distributions and $|r|\geq1$; (4) $X$ follows some beta distributions and $r$ is taken from some interval. In particular, if $S$ is a standard semicircular element  then $|S|^r$ is freely infinitely divisible for $r\in(-\infty,0]\cup[2,\infty)$.  Also we consider the symmetrization of the above probability measures, and in particular show that $|S|^r \,\sign(S)$ is freely infinitely divisible for $r\geq2$. Therefore $S^n$ is freely infinitely divisible for every $n\in\N$. 
The results on free Poisson and semicircular random variables have a good correspondence with classical ID properties of powers of gamma and normal random variables.  
\end{abstract}



%\tableofcontents

\section{Introduction}
In classical probability, people have tried to understand if infinite divisibility (ID) can be preserved by powers, products or quotients of (independent) random variables (rvs). Usually the L\'evy--Khintchine representation is not useful for this purpose and alternative new ideas are required. One class that behaves well with respect to powers and products is mixtures of exponential distributions ($\ME$), i.e.\ rvs of the form $E X$ where $E,X$ are independent, $E$ follows an exponential distribution and $X \geq0$. The Goldie--Steutel theorem says that the class $\ME$ is a subset of the class $\ID$ (see \cite{Gol67} and \cite{Ste67}). In this class we have the implication
\be
 \text{$X\sim \ME$  $\Rightarrow$ $X^r\sim\ME$ for any $r\geq1$},  
\ee
and also $\ME$ is closed under the product of independent rvs. 

Quite a successful class in the theory of ID distributions is HCM (hyperbolically completely monotone) distributions \cite{Bon92}. It is known that $\HCM \subset \ID$ and 
\be
 \text{$X\sim \HCM$  $\Rightarrow$ $X^r\sim\HCM$ for any $|r|\geq1$}, 
\ee
see \cite[p.\ 69]{Bon92}. The class HCM moreover satisfies that if $X,Y\sim\HCM$ and are independent then $X Y, X/Y \sim \HCM$ \cite[Theorem 5.1.1]{Bon92}. Prior to the appearance of HCM, Thorin \cite{Tho77a,Tho77b} introduced a class $\GGC$ (generalized gamma convolutions) which contains HCM \cite[Theorem 5.1.2]{Bon92}. Bondesson recently proved that if $X,Y \sim \GGC$ and $X,Y$ are independent then $X Y\sim\GGC$ \cite{Bon13}. He also conjectured that $X\sim \GGC$ implies that $X^r\sim\GGC$ for any $r\geq1$, which is still open. The class $\GGC$ is closed  with respect to the addition of independent rvs, while  $\ME$ and $\HCM$ are not. It is worth mentioning that Shanbhag et al.\ \cite{SPS77} proved a related negative result that the product of two independent positive ID rvs is not always ID.

In free probability, the class of the free regular ($\FR$) distributions, i.e.\ the laws of nonnegative free L\'evy processes, is closed with respect to the product $\sqrt{X}Y \sqrt{Y}$ where $X,Y$ are free \cite[Theorem 1]{AHS13}. However, little is known on powers of rvs except that Arizmendi et al.\ showed that if $X\sim \FID$ is even (i.e.\ having a symmetric distribution) then $X^2\sim\FR$ \cite{AHS13}.  The main purpose of this paper is to consider the free infinite divisibility (FID) or more strongly free regularity of powers of rvs. 

In this paper we will focus on several examples including free Poisson and semicircular rvs and prove that: (1) If $X$ follows a free Poisson distribution without an atom at 0, then $X^r\sim \FID$ for any $r\in(-\infty,0]\cup[1,\infty)$; (2) If $X$ follows a free Poisson distribution with an atom at 0, then $X^r \sim\FID$ for any $r\geq1$; (3) If $X$ follows some mixtures of (classical) HCM distributions then $X^r\sim\FR$ for any $|r|\geq1$; (4) If $X$ follows some beta distributions then $X^r\sim\FID$ for $r$ in some interval. Our result has a consequence that $|S|^r \sim\FID$ when $S$ is the standard semicircular element and $r\in(-\infty,0]\cup[2,\infty)$. We will also consider the symmetrization of powers of beta rvs and mixtures of HCM rvs, and in particular,  show that $|S|^r \,\sign(S)\sim\FID$ for $r\geq2$. These results imply that $S^n$ are FID for all $n\in\N$. Our result on HCM distributions generalizes \cite[Proposition 4.21(1)]{AH} where mixtures of positive Boolean stable laws are shown to be FID and part of \cite[Theorem 1.2(3)]{Has14} where beta distributions of the 2nd kind are shown to be FID. 

The proofs depend on the complex-analytic method which has been developed recently. Until around 2010, one could prove the FID property of a given probability measure only when the $R$-transform \cite{Voi86} or Voiculescu transform \cite{BV93} is explicit. There had been no way to prove the FID property if the $R$-transform is not explicit. Actually many probability distributions used in classical probability theory do not have explicit $R$-transforms, e.g.\ normal distributions, gamma distributions and beta distributions. By contrast, there are lots of methods in classical probability to show that a probability measure is ID, even if its characteristic function is not explicit. 

In 2011, Belinschi et al.\ changed this situation and they gave the first nontrivial FID probability measure: the normal distribution is FID \cite{BBLS11}. The proof is based on the complex analysis of the Cauchy transform (but there is combinatorial background). Since then, several other people developed the complex-analytic method. Now many nontrivial distributions are known to be FID: Anshelevich et al.\ showed that the $q$-normal distribution is FID for $q\in[0,1]$ \cite{ABBL10} and the author proved that beta distributions of the 1st kind and 2nd kind, gamma, inverse gamma and Student distributions are FID for many parameters \cite{Has14}. Other results can be found in \cite{AB13,AH13a,AH14,AH,AHS13,BH13}. These examples suggest that the intersection of ID and FID is rich. Almost all of these FID distributions belong to a further subclass UI that was introduced in \cite{AH13a}. With this class UI we are able to show the FID property of a given probability measure without knowing the explicit $R$-transform or free L\'evy--Khintchine representation. This class plays an important role in the present paper too.

 This paper contains two sections besides this section. In Section \ref{sec2} we will introduce basic notations and concepts including the classes ID, ME, GGC, HCM, FID, FR, UI and  probability measures to be treated. In Section \ref{sec3} we will state the main results rigorously and then prove them. Section \ref{sec4} explains similarity between our results on free Poisson and semicircle distributions and the classical results on gamma and normal distributions. Some conjectures are proposed based on this similarity.

\section{Preliminaries}\label{sec2}
Some general notations in this paper are summarized below. 
\begin{enumerate}[\rm(1)] 
\item $\PM_+$ is the set of (Borel) probability measures on $[0,\infty)$. 

\item For a classical or noncommutative rv $X$ and a probability measure $\mu$ on $\R$, the notation $X\sim \mu$ means that the rv $X$ follows the law $\mu$.  A similar notation is used for a subclass of probability measures: for example $X\sim \PM_+$ means that $X \sim\mu$ for some $\mu\in \PM_+.$

\item The function $z^p$ is the principal value defined on $\C \setminus(-\infty,0]$. 

\item The function $\arg(z)$ is the argument of $z$ defined in $\C\setminus(-\infty,0]$, taking values in $(-\pi,\pi)$. We also use another argument $\arg_{I}(z)$ taking values in an interval $I$. 
\end{enumerate}

\subsection{ID distributions and subclasses} 
Infinitely divisible distributions and subclasses are summarized here. The reader is referred to \cite{SVH03,Bon92} for more information on this section.  

A probability measure on $\R$ is said to be {\it infinitely divisible} (ID) if it has an $n^{\rm th}$ convolution root for any $n\in\N$. The class of ID distributions is denoted by $\ID$ (and this kind of notations will be adapted to other classes of probability measures too).  

Let $\g(p,\theta)$ be the gamma distribution 
\be
\frac{1}{\theta^{p}\Gamma( p)} x^{p-1} e^{-x/\theta}\,1_{(0,\infty)}(x)\,dx,\qquad p,\theta>0,  
\ee
where $\theta$ corresponds to the scaling. A probability measure $\mu\in\PM_+$ is called a {\it mixture of exponential distributions} ($\ME$) if there exists $\nu \in \PM_+$ such that $\mu= \g(1,1) \circledast \nu$, where $\circledast$ is classical multiplicative convolution. An equivalent definition is that $\mu$ is of the form $w \delta_0 + f(x) dx$, where $w \in[0,1]$ and $f$ is a $C^\infty$ function on $(0,\infty)$ which is completely monotone, i.e.\ $(-1)^{n}f^{(n)}\geq0$ for any $n\in\N\cup\{0\}$. It is known that $\ME \subset \ID$, called the Goldie--Steutel theorem. 

A probability measure on $[0,\infty)$ is called a {\it generalized gamma convolution} ($\GGC$) if it is in the weak closure of the set 
\be
\{\g(p_1,\theta_1)\ast \cdots \ast \g(p_n,\theta_n)\mid p_k,\theta_k>0, k=1,\dots,n, n\in\N\}, 
\ee
that is, the class of $\GGC$s is the smallest subclass of $\ID$ that contains all the gamma distributions and that is closed under convolution and weak limits.  

A pdf (probability density function) $f: (0,\infty) \to [0,\infty)$ is said to be {\it hyperbolically completely monotone} (HCM) if for each $u>0$ the map $w \mapsto f(u v) f(u/v)$ is completely monotone as a function of $w=v+1/v$. A probability distribution on $(0,\infty)$ is called an HCM distribution if it has an HCM pdf. 
It turns out that any HCM pdf is the pointwise limit of pdfs of the form
\be\label{pdf HCM}
C\cdot x^{p-1} \prod_{k=1}^n (t_k +x)^{-\gamma_k}, \qquad C, p, t_k, \gamma_k>0,  p<\sum_{k=1}^n\gamma_k, n\in\N. 
\ee 
 This limiting procedure gives us the representation of an HCM pdf 
\be
f(x) = C\cdot x^{\alpha-1} \exp\left(-b_1 x-\frac{b_2}{x}+ \int_{(1,\infty)} \log\left(\frac{t+1}{t+x}\right)\, \Gamma_1(dt) + \int_{[1,\infty)} \log\left(\frac{t+1}{t+1/x}\right)\, \Gamma_2(dt)\right)
\ee
where $\alpha\in\R, b_1,b_2 \geq0$ and $\Gamma_1, \Gamma_2$ are measures on $(1,\infty)$ and $[1,\infty)$, respectively, such that $\int (1+t)^{-1}\Gamma_k(dt)<\infty$. However these conditions on the parameters are not sufficient to ensure the integrability of $f$. 
The author does not know how to write down necessary and sufficient conditions for $\int_{(0,\infty)} f(x)\,dx <\infty$ in terms of the parameters $\alpha,b_1,b_2, \Gamma_1,\Gamma_2$. An important fact in the theory of ID distributions is that $\HCM \subset \GGC$.

\subsection{FID distributions and subclasses}
A probability measure on $\R$ is said to be {\it freely infinitely divisible} (FID) if it has an $n^{\rm th}$ free convolution root for each $n\in\N$. Let $\FID$ be the set of FID distributions on $\R$. 
Basic results on the class FID were established in \cite{BV93}. Connections between the class $\FID$ and free L\'evy processes were investigated in \cite{Bia98,BNT02,BNT05,BNT06,AHS13}. Bercovici and Pata clarified how the FID distributions appear as the limit of the sum of free i.i.d.\ rvs \cite{BP99}.

We say that $\mu\in\PM_+$ is {\it free regular} (FR) if $\mu$ is FID and $\mu^{\boxplus t} \in \PM_+$ for all $t>0$. This notion was introduced in \cite{PAS12} in terms of the Bercovici--Pata bijection and then further developed in \cite{AHS13} in terms nonnegative free L\'evy processes. The set of free regular distributions is denoted by $\FR$. The class $\FR$ is closed with respect to the weak convergence, see \cite[Proposition 25]{AHS13}. 
A probability measure in $\FID\cap\PM_+$ may not be free regular, but we have a criterion. 
\begin{lem}[Theorem 13 in \cite{AHS13}]\label{lem3}
Suppose $\mu\in\FID\cap \PM_+$ and $\mu$ satisfies either $(i)$ $\mu(\{0 \}) >0$, or $(ii)$ $\mu(\{0 \}) =0$ and $\int_{(0,1)} x^{-1}\,\mu(dx) = \infty$.  Then $\mu\in\FR$. 
\end{lem}

There is a useful subclass of $\FID$, called $\UI$. The idea already appeared implicitly in \cite{BBLS11} and the explicit definition was given in \cite{AH13a}. 
The following form of definition is in \cite{BH13}. To define the class $\UI$, let $G_\mu$ (or $G_X$ if $X\sim \mu$) denote the Cauchy transform 
\be
G_\mu(z)=\int_\R \frac{1}{z-x}\,\mu(dx),\qquad z\in \C^+,  
\ee
where $\C^\pm$ denote the complex upper and lower half-planes respectively. 
In \cite {BV93} Bercovici and Voiculescu proved that for any $a>0$, there exist $\lambda, M, b>0$ such that $G_\mu$ is univalent in the truncated cone 
\be\label{cone}
\Gamma_{\lambda,M}=\{z\in\C^+\mid \lambda |\re(z)| <\im(z), \im(z) >M\}
\ee
and $G_\mu(\Gamma_{\lambda, M})$ contains the triangular domain 
\be
\Delta_{a,b} =\{w\in \C^-\mid a |\re(w)| < -\im(w), \im(w)>-b\}. 
\ee
So we may define a right compositional inverse $G_\mu^{-1}$ in $\Delta_{a,b}$. 
\begin{defi}[Definition 5.1 in \cite{AHS13}]
A probability measure $\mu$ on $\R$ is said to be in class $\UI$ (standing for univalent inverse Cauchy transform) if the right inverse map $G_\mu^{-1}$, originally defined in a triangular domain $\Delta_{a,b} \subset \C^-$, has univalent analytic continuation to $\C^-$. 
\end{defi}
The importance of this class is based on the following lemma proved in \cite[Proposition 5.2 and p.\ 2763]{AH13a}. 
\begin{lem}\label{lem1}
$\UI\subset\FID$. Moreover, $\UI$ is w-closed (i.e.\ closed with respect to weak convergence). 
\end{lem}
To consider symmetric distributions, a symmetric version of $\UI$ is useful. The following definition is equivalent to \cite[Definition 2.5(2)]{Has14}. 
\begin{defi} \label{UIS}
A symmetric probability measure $\mu$ is said to be in class $\UIS$ if (a) the right inverse map $G_\mu^{-1}$, defined in a domain $\Delta_{a,b}$, has univalent analytic continuation to a neighborhood of $\I(-\infty,0)$, and (b) the right inverse $G_\mu^{-1}$, defined in the domain $\Delta_{a,b}$, has univalent analytic continuation to $\C^- \cap \I\C^-$ such that $G^{-1}_\mu(\C^- \cap \I\C^-)\subset \C^- \cup \I\C^-$. 
\end{defi}
Condition (a) cannot be dropped from the definition: $G_{\frac{1}{2}(\delta_{-1}+\delta_1)}^{-1}(z)= \frac{1+\sqrt{1+4 z^2}}{2 z}$  satisfies condition (b) but not condition (a).

If a symmetric probability measure $\mu$ belongs to $\UI$ then $\mu \in\UIS$ by definition, but the converse is not true. In fact a probability measure $\mu\in\UIS$ belongs to $\UI$ iff $G_\mu^{-1}(\C^-\cap\I\C^-) \subset \I\C^-$.  

The analogue of Lemma \ref{lem1} holds true. 
\begin{lem}
$\UIS \subset \FID$. Moreover $\UIS$ is w-closed. 
\end{lem}
\begin{proof}
The first claim was proved in \cite[Lemma 2.7]{Has14}, but the second claim is new. 
Suppose that $\mu_n \in\UIS$ and $\mu_n\to\mu$ (which implies that $\mu\in\FID$). By \cite[Theorem 3.8]{BNT02} the Voiculescu transform $\varphi_{\mu_n}$ converges to $\varphi_\mu$ uniformly on each compact subset of $\C^+$ and therefore so does $G_{\mu_n}^{-1}(z) = \varphi_{\mu_n}(z^{-1})+z^{-1}$ to $G_{\mu}^{-1}(z) =\varphi_{\mu}(z^{-1})+z^{-1}$ on each compact subset of $\C^-$. Since $G_\mu^{-1}$ is not a constant, it is univalent in $\C^- \cap \I\C^-$ by Hurwitz's theorem. Taking the limit we have $G^{-1}_\mu(\C^- \cap \I\C^-)\subset \overline{\C^- \cup \I\C^-}$, but since $G_\mu^{-1}$ is analytic and not a constant, $G^{-1}_\mu(\C^- \cap \I\C^-)$ is an open set by the open mapping theorem, so in fact $G^{-1}_\mu(\C^- \cap \I\C^-)$ is contained in $\C^- \cup \I\C^-$. 

Hurwitz's theorem is not useful for proving the univalence around $\I(-\infty,0)$ since the neighborhoods of $\I(-\infty,0)$ may shrink as $n\to\infty$, and so an alternative idea is needed. Our assumptions imply that $G_{\mu_n}^{-1}$ are locally univalent, i.e.\ $(G_{\mu_n}^{-1})'(z) \neq0$ for all $z\in\C^-$ (note that $G_{\mu_n}^{-1}$ has symmetry with respect to the $y$-axis). Since a zero of an analytic function changes continuously with respect to the local uniform topology of analytic functions, $(G_\mu^{-1})'$ does not have a zero in $\C^-$; otherwise $(G_{\mu_n}^{-1})'$ would have a zero for large $n$. In particular  $(G_\mu^{-1})'(z)\neq0$ in $\I(-\infty,0)$. Since $\mu$ is symmetric, $G_\mu^{-1}(\I(-\infty,0))\subset \I\R$ and hence $(G_\mu^{-1})'(\I(-\infty,0))\subset\R$. Since the derivative $(G_\mu^{-1})'(z)$ does not have a zero, the sign does not change (in fact we can show that $(G_\mu^{-1})'(z)>0$ in $\I(-\infty,0)$). 
Therefore $G_\mu^{-1}$ is univalent in a neighborhood of $\I(-\infty,0)$ by applying Noshiro--Warschawski's theorem \cite[Theorem 12]{Nos34}, \cite[Lemma 1]{War35} to a neighborhood of $\I(-\infty,0)$ where $\re((G_\mu^{-1})'(z))>0$.  
\end{proof}

\subsection{Probability measures to be treated} 
We introduce several (classes of) probability measures to be treated in this paper. 

(1) The semicircle distribution $\WS(m, \sigma^2)$ is the probability measure with pdf 
\be
\frac{1}{2\pi \sigma^2} \sqrt{4\sigma^2-(x-m)^2}\, 1_{(m-2\sigma, m+2\sigma)}(x), \qquad m\in\R,\sigma>0. 
\ee
By using the $R$-transform one can compute the inverse Cauchy transform 
\be
G_{\WS(m,\sigma^2)}^{-1}(z) = m + \sigma^2 z + \frac{1}{z}. 
\ee One can prove that $G_{\WS(m,\sigma^2)}^{-1}$ is a bijection from $\C^-$ onto $\C^+\cup(m-2\sigma,m+2\sigma)\cup\C^-$, so that $\WS(m,\sigma^2) \in \UI$ and in particular $\WS(0,\sigma^2)\in\UIS$.

(2) The free Poisson distribution (or Marchenko-Pastur distribution) $\MP(p, \theta)$ is 
\be
\begin{split}
& \max\{1-p,0\} \delta_0 + \frac{\sqrt{(\theta(1+\sqrt{p})^2-x)(x-\theta(1-\sqrt{p})^2)}}{2\pi\theta x}\,1_{(\theta(1-\sqrt{p})^2, \theta(1+\sqrt{p})^2)}(x)\,d x, 
\end{split}
\ee
where $p,\theta>0$. The parameter $\theta$ stands for the scale parameter. Since $\MP(p,\theta)^{\boxplus t} = \MP(p t, \theta) \in \PM_+$ for any $t>0$, $\MP(p,\theta)$ is free regular. The inverse Cauchy transform is given by 
\be
G_{\MP(p,\theta)}^{-1}(z) = \frac{1}{z} + \frac{p\theta}{1-\theta z}. 
\ee
One can show that $G_{\MP(p,\theta)}^{-1}$ is a bijection from $\C^-$ onto $\C^+\cup(\theta(\sqrt{p}-1)^2,\theta(\sqrt{p}+1)^2)\cup\C^-$ and so $\MP(p,\theta) \in\UI$. 

(3) The beta distribution $\Beta(p,q)$ is the probability measure with pdf 
\be
\frac{1}{B(p,q)} \cdot x^{p-1}(1-x)^{q-1}\,1_{(0,1)}(x),\qquad p,q>0. 
\ee
The beta distribution belongs to the class $\UI$ if 
(i) $p, q \geq \frac{3}{2}$,  (ii) $0<p\leq \frac{1}{2},~p+q \geq 2$ or (iii) $0<q\leq \frac{1}{2},~p+q \geq 2$ (see \cite{Has14}). 

(4) The positive Boolean (strictly) stable distribution $\textbf{b}_\alpha$, introduced by Speicher and Woroudi \cite{SW97}, is the distribution with pdf 
\be
\frac{\sin(\alpha \pi)}{\pi} \cdot \frac{x^{\alpha-1}}{x^{2\alpha} +2(\cos \pi\alpha) x^\alpha+1}\, 1_{(0,\infty)}(x),\qquad \alpha \in(0,1). 
\ee 
We then consider the classical mixtures of positive Boolean stable distributions 
\be
\B_\alpha=\{\mu \circledast \bs_\alpha \mid \mu \in \PM_+ \}, \qquad \alpha \in(0,1), 
\ee
where $\circledast$ is classical multiplicative convolution, i.e.\ $X Y\sim \mu\circledast \nu$ when $X\sim\mu, Y\sim\nu$ and $X,Y$ are independent. This class was introduced and investigated in \cite{AH}. It is known that $\B_\alpha \subset \B_\beta$ when $0<\alpha \leq \beta<1$ and $\B_{1/2} \subset \UI\cap \FR\cap\ME$. 
The property $\B_{1/2} \subset\FR$ is not explicitly stated in \cite{AH}, but it is a consequence of the fact $\B_{1/2} \subset \FID \cap \PM_+$ proved in \cite[Proposition 4.21]{AH}, the fact that any measure in $\B_{1/2}\setminus\{\delta_0\}$ has a pdf which diverges to infinity at 0, and Lemma \ref{lem3} above. 

(5) We consider a further subclass of HCM distributions with pdf 
\be\label{sub HCM}
f(x) = C\cdot x^{p-1} \exp\left(\int_{(0,\infty)} \log\left(\frac{1}{t+x}\right) \Gamma(dt)\right), 
\ee
where $\Gamma$ is a finite measure on $(0,\infty)$, $0<p < \Gamma((0,\infty))$ and $\int_{(0,\infty)} |\log t | \,\Gamma(d t)<\infty$ (these conditions ensure the integrability of $f$ and hence we can take the normalizing constant $C>0$). This subclass is a natural generalization of pdfs of the form \eqref{pdf HCM}, but does not cover all HCM pdfs. 
This pdf is of the form of Markov-Krein transform \cite{Ker98}.

\section{Main results}\label{sec3}
\subsection{Statements}
Now we are ready to state the main theorems. Independence of random variables means classical independence below. 
\begin{thm}\label{thm1}
Suppose $X\sim \HCM$ whose pdf is of the form \eqref{sub HCM} and satisfies 
\be\label{cond}
0< p \leq \frac{1}{2},  \qquad 0<-p+ \Gamma((0,\infty)) \leq  \frac{1}{2}.
\ee 
If $W \geq 0$ is a rv  independent of $X$ and $|r|\geq1$,  then $W X^r \sim \FR\cap\UI$. The case $r\leq-1$ can be considered only when $W>0$ almost surely.  Moreover, $X^{-1}$ also has a pdf of the form \eqref{sub HCM} that satisfies condition \eqref{cond}. 
\end{thm}
\begin{rem} If one uses \cite[Theorems 4.1.1, 4.1.4]{Bon92} and the factorization of a gamma rv as a product of a beta rv and a gamma rv on p.\ 14 of \cite{Bon92}, one can show that $W X^r\sim\ME$ and thus we get $W X^r \sim \FR \cap\UI\cap \ME$.  
\end{rem}
Since the Boolean stable law $\bs_{1/2}$ has the pdf $\pi^{-1}x^{-1/2}(x +1)^{-1}$ which satisfies the assumptions of Theorem \ref{thm1}, we have  
\begin{cor} \label{cor1} 
If $X \sim \B_{1/2}$ and $|r|\geq1$ then $X^r \sim \FR\cap\UI$ $($$r\leq-1$ can be considered only when the mixing measure does not have an atom at $0$$)$.   
\end{cor}

\begin{thm}\label{thm2}
Suppose that $X\sim \Beta(p,q)$. If $r \geq1$, $q \in [\frac{3}{2}, \frac{5}{2}]$, $0<2 p \leq r$ and $(q-\frac{3}{2})r \leq p+q-1 \leq r$, then $X^r \sim \FR\cap\UI$.  
\end{thm}
Since $\Beta(\frac{1}{2},\frac{3}{2})=\MP(1,\frac{1}{4})$, we conclude that if $X\sim\MP(1,1)$ then $X^r \sim\FR\cap \UI$ for $r\geq1$. We will prove FID property for a larger class of free Poissons.

\begin{thm}\footnote{The published version states that $X^r$ is also FR in both cases \eqref{thm3-1} and \eqref{thm3-2}, but the proof is not correct. If $X \sim \MP(1,1)$ then $X^r \sim \FR$ for $r\geq1$ as mentioned above, but  for other cases whether $X^r \in \FR$ or not is unclear. }\label{thm3} 
Suppose that $X\sim \MP(p,1)$ and $p>0$. 
\begin{enumerate}[\rm(1)]
\item\label{thm3-1} If $p \geq 1$ and $r \in(-\infty,0]\cup[1,\infty)$,  then $X^r \sim \UI$.  
\item\label{thm3-2} If $0<p < 1$ and $r \geq1$,  then $X^r \sim \UI$.  
\end{enumerate}
\end{thm}
We avoided the negative powers for $0<p<1$ since $\MP(p,1)$ has an atom at 0.  

Since $S \sim \WS(0,1)$ implies $S^2 \sim \MP(1,1)$, the following result is immediate.  
\begin{cor}\label{Power WS} 
If $S\sim \WS(0,1)$ and $r \in(-\infty,0]\cup[2,\infty)$ then $|S|^{r} \sim \UI$. If $r\geq2$ we also have $|S|^{r} \sim \FR$.
\end{cor}
\begin{rem} 
The result for $r=4$ was proved in \cite{AHS13}. Recently Chistyakov and Lehner showed the $r=6$ case (in private communication). A related but negative result is that $(S+a)^2\not\sim\FID$ for all $a\neq0$ \cite{Eis12}.  
\end{rem}

We consider the symmetrized versions. 
\begin{thm}\label{Sym Power Beta}
\begin{enumerate}[\rm(1)]
\item Suppose $X\sim \HCM$ satisfying condition \eqref{cond} in Theorem \ref{thm1}. If $W$ is independent of $X$ and has a symmetric distribution, then $W X^r \sim \UIS$ for $|r|\geq1$.  
\item\label{Sym Beta} Suppose that $B,X$ are independent and $B\sim \frac{1}{2}(\delta_{-1}+\delta_1), X\sim \Beta(p,q)$. Under the assumptions on $p,q,r$ in Theorem \ref{thm2}, we have $B X^r \sim \UIS$. 
\item\label{Sym Beta not FID} Suppose that $B,X$ are independent and $B\sim \frac{1}{2}(\delta_{-1}+\delta_1), X\sim \Beta(p,q)$. If either $(i)$ $p>r>0$ or $(ii)$ $r\leq0$ then we have $B X^r \not\sim \FID$. 
\end{enumerate}
\end{thm}
The case $(p,q)=(1/2,3/2)$ in \eqref{Sym Beta} and \eqref{Sym Beta not FID} gives us 
\begin{cor}\label{Sym Power WS}
\begin{enumerate}[\rm(1)]
\item If $S\sim \WS(0,1)$ and $r \geq2$ then $|S|^{r}\,\sign(S) \sim \UIS$. 
\item If $S\sim \WS(0,1)$ and $r<1$ then $|S|^r\, \sign(S) \not\sim \FID$.
\end{enumerate}
\end{cor}
\begin{rem} 
Corollary \ref{Power WS} and Corollary \ref{Sym Power WS} imply that $S^n$ are FID for all $n\in\N$. In the special case $n=2$, the distribution of $|S|^2\, \sign(S)$ is the symmetrized beta distribution with parameters $1/2,3/2$ and Arizmendi et al.\ already proved that it is FID \cite[Proposition 11]{ABNPA10}.
\end{rem}

{\bf Discussions and further remarks on main theorems.}
(1) In Theorem \ref{Sym Power Beta} only  the symmetric Bernoulli rv $B$ is considered for the mixing of beta rvs, while a general $W$ with a symmetric distribution appears in the HCM case. This Bernoulli distribution cannot be generalized to arbitrary symmetric distributions. Actually if $W$ has a symmetric discrete distribution $\mu_W$ whose support has cardinality $\geq 4$, then $W X$ is not FID for any beta rv $X$. The proof is as follows. The pdf of $W X$ is positive at $c:=\min\{x\in\text{supp}(\mu_W)\mid x>0\}$ but is not real analytic at $c$. This shows that the distribution of $W X$ is not FID by \cite[Proposition 5.1]{BB05}. For a similar reason, one cannot take a (positive) scale mixing of beta or free Poisson rvs in Theorem \ref{thm2} or in Theorem \ref{thm3}.

(2) Related to Theorem \ref{Sym Power Beta}, a natural question is the symmetrization of powers of free Poisson rvs, i.e.\ the law of $B X^r$ where $B,X$ are independent, $B\sim \frac{1}{2}(\delta_{-1}+\delta_1)$ and $X\sim \MP(p,1)$. 
\begin{itemize}
\item If $p=1$ and $r\geq1$ then $B X^r\sim \UIS$ from Corollary \ref{Sym Power WS}(1). 

\item If $p>1$ and $r\in\R$ then $B X^r$ is not FID from Lemma \ref{lem pdf 0} that we will show later. 
\item If $p<1$ and $r\geq1$ then it is not known whether $B X^r$ is FID or not. Actually one can check condition (b) in Definition \ref{UIS} similarly to the proof of Theorem \ref{Sym Power Beta}, but it is not clear how to show condition (a).  
\end{itemize}

(3) In Theorems \ref{thm1}, \ref{thm2}, \ref{Sym Power Beta}, the assumptions on the parameters of beta and HCM distributions  may seem too restrictive. Weakening these assumptions is left to future research.

\subsection{Proofs}
The integral form of the Cauchy transform gives us an analytic function defined outside the support of $\mu$, and we denote it by $\widetilde{G}_\mu(z)$:  
\be
\widetilde{G}_\mu(z) =\int_\R \frac{1}{z-x}\,\mu(dx),\qquad z\in \C \setminus \text{supp}(\mu).   
\ee
In the study of FID distributions, more important is the analytic continuation of the Cauchy transform $G_\mu$ from $\C^+$ into $\C^-$ passing through the support of $\mu$. This analytic continuation is possible when the pdf is real analytic and the explicit formula can be given in terms of $\widetilde{G}_\mu(z)$ and the pdf. We state the result in a slightly general form where complex measures are allowed, but the proof is similar to that of \cite[Proposition 4.1]{Has14}. Note that $G_\sigma, \widetilde{G}_\sigma$ can be defined for complex measures $\sigma$ by linearity. 
\begin{lem}\label{lem analytic cauchy}
Let $I \subset \R$ be an open interval. Suppose that $f$ is analytic in a neighborhood of $I \cup \C^-$ and that $f$ is integrable on $I$ with respect to the Lebesgue measure. We define the complex measure $\sigma(dx):= f(x)1_{I}(x)dx$. Then the Cauchy transform $G_\sigma$ defined on $\C^+$ has analytic continuation to $\C^+ \cup I \cup \C^-$, which we denote by the same symbol $G_\sigma$, and 
\be\label{eq analytic cauchy} 
G_\sigma(z) 
=  \widetilde{G}_\sigma(z) - 2\pi \I f(z), \qquad z \in \C^-. 
\ee
\end{lem}

The following lemma gives a crucial idea for showing the main theorems. The idea of the proof is to take the curve $\gamma$ in \cite[Proposition 2.1]{BH13} to be the one that starts from $-\infty +\I0$, then goes to 0, turns $180^{\circ}$ around 0 and then goes to $-\infty-\I0$. This curve is useful since we can easily compute the boundary value of $G_\mu$ on $(-\infty,0) -\I0$ thanks to formula \eqref{eq analytic cauchy}, so that we can check condition (B) in \cite[Proposition 2.1]{BH13}. 

\begin{lem}\label{lem4} Let $\mu$ be a probability measure on $[0,\infty)$ which has a pdf of the form 
\be\label{form}
f(x)= x^{p-1} g(x), \qquad x>0, 
\ee
where 
\begin{enumerate}[\rm({A}1)]
\item\label{AS1} $0<p\leq \frac{1}{2}$;  
\item\label{AS2} $g$ is the restriction of an analytic function (also denoted by $g$) defined in $\{z\in\C\setminus\{0\}\mid \arg(z) \in(-\pi,\theta_0)\}$ for some $\theta_0\in(0,\pi)$ and the restriction $g|_{\C^-}$ extends to a continuous function on $\C^-\cup(-\infty,0)$; 
%\item\label{AS3} $g$ has analytic continuation to $\{|z|<\delta \mid z\neq0, \arg_{(-3\pi/2,\pi/2)}(z) \in(-\pi-\theta_0,\theta_0)\}$ for some $\theta_0\in(0,\pi/4), \delta>0$; 
\item\label{AS4} $\lim_{z\to0, \arg(z)\in(-\pi,\theta_0)}g(z)>0$; 
\item\label{AS5} $\lim_{|z|\to\infty, \arg(z)\in(-\pi,0)} z^{\gamma} f(z)=0$ for some $\gamma>0$; 
%\item\label{AS6} The function $g|_{\C^-}$ extends to a continuous function on $\C^- \cup (-\infty, 0]$. We denote the continuous boundary function by $g(x-\I0)$, $x<0$; 
\item\label{AS7} $\re(f(x-\I0)) \leq0$ for $x<0$. 
\end{enumerate} 
Then $\rho \circledast \mu \in \FR \cap \UI$ for all $\rho \in \PM_+$. 
%Moreover, we have the asymptotic behavior 
\end{lem}
\begin{proof} We first assume that $\rho =\delta_1$ and $0<p<1/2$ and later drop these assumptions. 
Lemma \ref{lem analytic cauchy} implies that the Cauchy transform $G_\mu$ has analytic continuation (denoted by $G_\mu$ too) to $\C \setminus (-\infty,0]$ and 
$G_\mu(z)$ is given by formula \eqref{eq analytic cauchy} in $\C^-$.  

In the following $\eta>0$ is supposed to be large and $\delta>0$ is supposed to be small. We consider curves: 
\begin{itemize}
\item $c_1$ is the real line segment from $-\eta+\I0$ to $-\delta+\I0$; 
\item $c_2$ is the clockwise circle $\delta e^{\I\theta}$ where $\theta$ starts from $\pi$ and ends with $-\pi$; 
\item $c_3$ is the line segment from $-\delta-\I0$ to $-\eta -\I0$; 
\item $c_4$ is the counterclockwise circle centered at 0, starting from $-\eta -\I0$ and stopping at $-\eta+\I0$. 
\end{itemize}
Note that the line segments $c_1$ and $c_3$ are meant to be different by considering a Riemannian surface. 
The left of Fig.\ \ref{fig 1} shows the directed closed curve consisting of $c_k, k=1,2,3,4$. 
Let $g_k$ be the image curve $G_\mu(c_k)$ for $k=1,2,3,4$. More precisely, the curves $g_1,g_3$ are defined by $G_\mu(c_1+\I0)$, $G_\mu(c_3-\I0)$ respectively, and hence $g_1$ lies on the negative real line.

\begin{figure}[htpb]
\begin{minipage}{0.5\hsize}
\begin{center}
\includegraphics[width=50mm,clip]{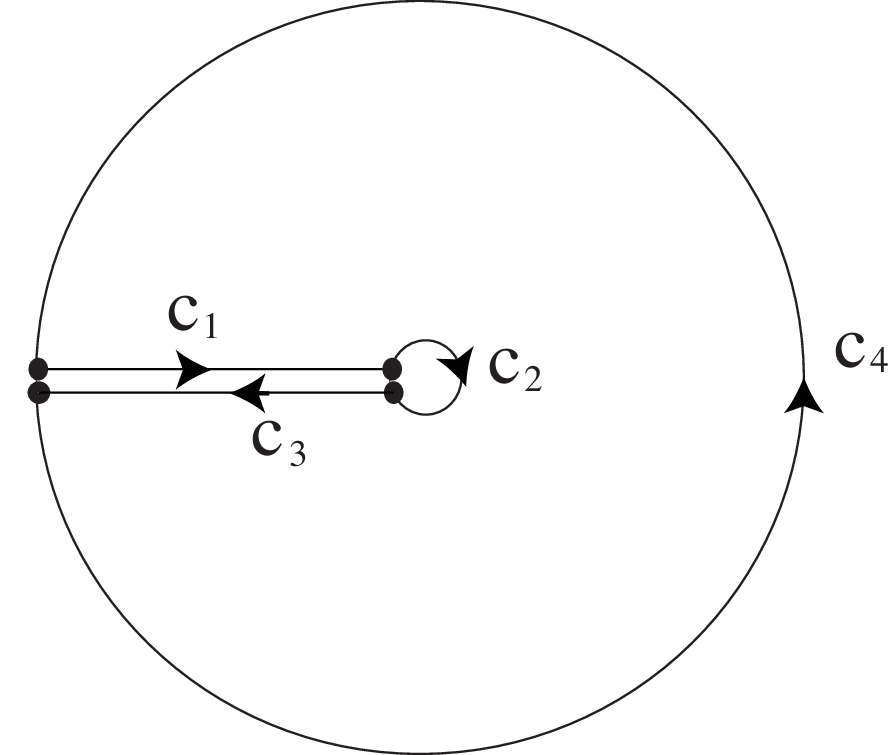}
\end{center}
  \end{minipage}
\begin{minipage}{0.5\hsize}
\begin{center}
\includegraphics[width=50mm,clip]{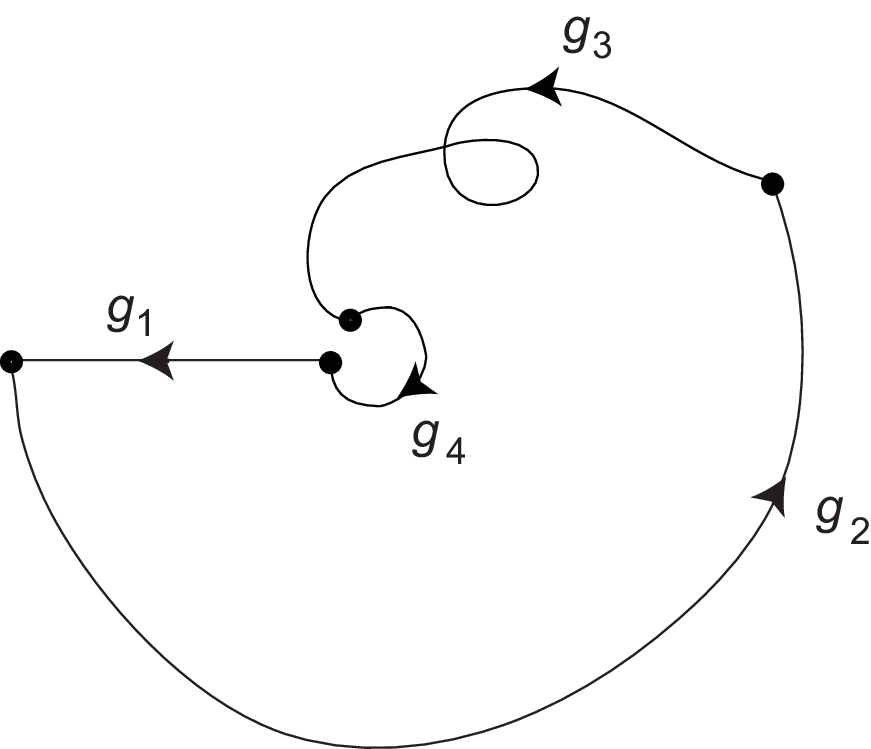}
\end{center}
\end{minipage}
\caption{The curves $c_k$ and $g_k$}\label{fig 1}
\end{figure}

Let $\epsilon>0$ be supposed to be small.  We claim the existence of $\eta>0$ large enough so that $|G_\mu(z)| < \epsilon$ for $|z|\geq\eta$,  $z\in \C \setminus (-\infty,0]$. This can be proved by dividing the region into the two parts $\arg(z)\in(-b,\pi)$ and $\arg(z)\in(-\pi,-b]$ where $b\in(0,\pi/2)$ is arbitrary. First, thanks to assumption (A\ref{AS5}) we may change the contour of the integral and obtain  
\be
G_\mu(z) = \int_{e^{-2\I b}(0,\infty)} \frac{1}{z-w}\,f(w)\,dw,\qquad  \arg(z)\in(-b,\pi).  
\ee
Using assumption (A\ref{AS5}),  we can show that $\sup_{ \theta\in(-b,\pi)}|G_\mu(re^{\I\theta})| \to0$ as $r\to\infty$. For $\arg(z)\in(-\pi,-b]$, we use \eqref{eq analytic cauchy} and assumption (A\ref{AS5}) to obtain $\sup_{\theta\in(-\pi,-b]}|G_\mu(r e^{\I\theta})| \to0$ as $r\to\infty$. This is what we claimed. Therefore, the curve $g_4$ is contained in the ball centered at 0 with radius $\epsilon$. 

From assumption (A\ref{AS7}) it holds that 
\be
\im(G_\mu (z)) =  \im(\widetilde{G}_\mu(z)) - 2\pi\, \im(\I f(z)) \geq 0,\qquad z\in [-\eta,-\delta]-\I0
\ee 
since $ \im(\widetilde{G}_\mu(z))=0$ in $(-\infty,0)-\I0$. Therefore the curve $g_3$ is contained in $\{w \mid \im(w) \geq0\}$. 

Assumptions (A\ref{AS2}),(A\ref{AS4}) allow us to use \cite[(5.6)]{Has14}. The assumptions of \cite[Theorem 5.1]{Has14} are a little different from the present case, but the proof is applicable to our case without a change, to conclude that there exists $a>0$ such that 
\be\label{eq cauchy near 0}
G_\mu(z) = -a (-z)^{p-1} + o(|z|^{p-1}),\qquad \text{as~} z\to0, -\pi <\arg(z) < \pi.
\ee
By asymptotics \eqref{eq cauchy near 0}, we can take $\delta \in(0, \left(a\epsilon/2\right)^{\frac{1}{1-p}})$ small enough so that 
\be
|G_\mu(\delta e^{\I\theta})+a(-\delta e^{\I\theta})^{p-1}| < \frac{a \delta^{p-1}}{2}
\ee
 uniformly on $\theta \in (-\pi,\pi)$. Therefore  $|G_\mu(\delta e^{\I\theta}) |> \frac{1}{2}a\delta^{p-1} > \epsilon^{-1}$, and so the distance between the curve $g_2$ and 0 is larger than $\epsilon^{-1}$. Since we have \eqref{eq cauchy near 0}, if $\delta,\epsilon>0$ are small enough then the final point of $g_2$ has an argument approximately equal to $\arg(-(-(-\delta-\I0))^{p-1})=\arg(-(\delta e^{-2\pi\I})^{p-1})=(1-2 p)\pi$, which is contained in $(0,\pi)$ because we assumed $p<1/2$.  
  
From the above arguments, every point of $D_\epsilon:=\{w \in\C^-\mid \epsilon<|w| < \epsilon^{-1}\}$ is surrounded by the closed curve $g_1 \cup g_2 \cup g_3\cup g_4$ exactly once. Hence we can define a univalent inverse function $G_\mu^{-1}$ in $D_\epsilon$. By analytic continuation, we can define a univalent inverse function $G_\mu^{-1}$ in $\C^-$ by letting $\epsilon\downarrow0$.  When $\epsilon>0$ is small, the bounded domain surrounded by $c_1\cup c_2\cup c_3\cup c_4$ has nonempty intersection with a truncated cone $\Gamma_{\lambda, M}$  where $G_\mu$ is univalent, and $G_\mu(\Gamma_{\lambda,M})$ has intersection with $D_\epsilon$, and hence our $G_\mu^{-1}$ coincides with the original inverse on their common domain.  Thus our $G_\mu^{-1}$ gives the desired analytic continuation, to conclude that $\mu \in\UI$. The case $p=1/2$ follows by approximation. 
 %Note that the final point of the curve to $-a(-\delta e^{\I(\pi-\beta)})^{p-1} = -a\epsilon^{p-1} e^{\I\beta(1-p)}$.  

Next we take a discrete measure $\rho =\sum_{k=1}^n \lambda_k \delta_{t_k}, \lambda_k,t_k>0, \sum_{k} \lambda_k=1$.  The multiplicative convolution $\rho\circledast \mu$ has the pdf 
\be\label{eq0000}
\begin{split}
f_\rho(x)&:=\sum_{k=1}^n \lambda_k t_k^{-1}f(x/t_k) 
= \sum_{k=1}^n  \lambda_k t_k^{-1} (x/t_k)^{p-1} g(x/t_k) \\
&= x^{p-1} g_\rho(x), 
\end{split}
\ee
where $ g_\rho(x):=\sum_{k=1}^n  \lambda_k t_k^{-p}g(x/t_k).$ The pdf $f_\rho$ satisfies all the conditions (A\ref{AS1})--(A\ref{AS7}) which follow from the conditions for $f$. 
Hence what we proved for $f$ applies to $f_\rho$ without a change, and hence $\rho\circledast \mu \in \UI$. Since $\lim_{x\downarrow0}g_\rho(x)>0$ by (A\ref{AS4}), the pdf $f_\rho$ satisfies $\lim_{x\downarrow0}f_\rho(x)=\infty$, so that $\rho \circledast \mu$ satisfies condition (ii) in Lemma \ref{lem3}, and so $\rho \circledast \mu\in \FR$. 
Finally, by using the w-closedness of $\UI$ and $\FR$ we can approximate a general $\rho$ by discrete measures to get the full result. 
\end{proof}
\begin{rem} 
The curve $g_1$ is a Jordan curve (i.e.\ a curve with no self-intersection) since $G_\mu(x+\I0)$ is decreasing on $(-\infty,0)$. 
Moreover, if $\eta>0$ is large and $\delta>0$ is small then we can show that $g_2,g_4$ are also Jordan curves (with e.g.\ Rouche's theorem). 
However these properties are not needed to prove the theorem. Also, we do not know if $g_3$ has a self-intersection or not, but it does not matter for the proof.  
What we need is only that $g_3$ is contained in $\C^+\cup\R$. 
\end{rem}

The result on powers of HCM rvs is a consequence of the above lemma. 
\begin{proof}[Proof of Theorem \ref{thm1}] 
We further assume that $X$ has a pdf of the form \eqref{pdf HCM} and also $r>1$. Then we define $s:=1/r$. The pdf of $X^r$ is now given by 
\be
f(x) = s C\cdot x^{p s-1} \prod_{k=1}^n (t_k +x^s)^{-\gamma_k},   
\ee 
where
\be\label{ass 1}
0<p \leq \frac{1}{2}, \qquad 0<-p +\sum_{k=1}^n\gamma_k \leq \frac{1}{2}. 
\ee
Then $f$ can be written in the form $x^{p s -1} g(x)$ where 
\be \label{eq0001}
g(x)=s C \prod_{k=1}^n (t_k +x^s)^{-\gamma_k}.
\ee
The pdf $f$ is of the form \eqref{form} in Lemma \ref{lem4} with $p$ now replaced by $p s \in (0,1/2)$, satisfying assumptions (A\ref{AS1})--(A\ref{AS5}). Indeed, assumption (A\ref{AS2}) holds since we can extend the function \eqref{eq0001} analytically by extending $x^s$ to the analytic function $z^s$ defined in a domain $\{r e^{\I \theta }: r >0, \theta \in (-\pi-\delta,\delta)\}$ for some small $\delta>0$ so that $z^s$ never be a negative real in the domain. Such a $\delta>0$ exists by our assumption $s \in(0,1)$. Then (A\ref{AS4}) follows immediately, and (A\ref{AS5}) can be proved by taking $\gamma>0$ such that $\gamma < 1-p s +\sum_{k=1}^n\gamma_k s$. 
 In order to check (A\ref{AS7}) we compute 
\be
f(x-\I0) = - s C |x|^{p s -1}e^{-\I\pi p s} \prod_{k=1}^n (t_k + |x|^se^{-\I \pi s})^{-\gamma_k},\qquad x<0.  
\ee
Since $\arg(t_k + |x|^s e^{-\I\pi s}) \in (-\pi s, 0)$, we have that $\arg(t_k + |x|^s e^{-\I\pi s})^{-\gamma_k} \in (0, \pi \gamma_k s)$ and hence $\arg(-f(x-\I0)) \in (-\pi p s ,-\pi p s  + \sum_{k=1}^n \pi \gamma_k s )$. Then \eqref{ass 1} implies that $\arg(-f(x-\I0)) \in (-\frac{\pi}{2}, \frac{\pi}{2})$ and hence assumption (A\ref{AS7}) holds true. Thus $X^r \sim \FR \cap \UI$, and we can take the limit $s \uparrow1$ to get the result for $s\in(0,1]$. 

The law of $X^{-1}$ has the pdf  
\be
 C\cdot x^{-p -1} \prod_{k=1}^n (t_k +x^{-1})^{-\gamma_k} =  C' \cdot x^{p' -1} \prod_{k=1}^n (t_k^{-1} +x)^{-\gamma_k}, 
\ee 
where $C'=C \prod_k t_k^{-\gamma_k}, p'= -p + \sum_{k} \gamma_k$. Our assumption \eqref{ass 1} guarantees that $p'\in(0,\frac{1}{2}]$ and $-p'  +\sum_k  \gamma_k=p \in (0,\frac{1}{2}]$. Therefore $X^{-1}$ also satisfies \eqref{ass 1} with $p$ replaced by $p'$ and so $X^{-r} \sim \FR \cap \UI$ for $r\geq1$.  

Finally we can approximate a pdf of the form \eqref{sub HCM} by pdfs of the form \eqref{pdf HCM} in the sense of pointwise convergence. By Scheff\'e's lemma, we have the weak convergence of probability measures, and hence the full result follows. 
\end{proof}

We then go to the proof of Theorem \ref{thm2} on powers of beta rvs. The idea is similar to the case of HCM rvs, but we need more elaboration since now we have to study the boundary behavior of the Cauchy transform on $(1,\infty)-\I0$ in addition to $(-\infty,0)-\I0$. 

\begin{proof}[Proof of Theorem \ref{thm2}]
We may further assume that $r>1, p+q-1< r , q\in[\frac{3}{2},2)\cup(2,\frac{5}{2}], 2 p< r$ since the general case can be recovered by approximation.  We define $s:=1/r \in(0,1)$. 
The pdf of $X^r$ is now given by 
\be\label{eqii}
f(x) = \frac{s}{B(p,q)} \cdot x^{p s-1} (1-x^s)^{q-1} =x^{p s-1}g(x),  
\ee 
where $g(x) = \frac{s}{B(p,q)} (1-x^s)^{q-1}$. 

\noindent
{\bf Step 1: Analysis of $G_{X^r}$ around $(-\infty,0)$}. We take the same curves $c_1,c_2,c_3$ depending on $\eta,\delta>0$ as in Lemma \ref{lem4} and the image curves $g_k=G_{X^r}(c_k),k=1,2,3$. The pdf $f$ is of the form \eqref{form} with $p$ replaced by $p s\in(0,1/2)$ and satisfies assumptions (A\ref{AS1}),(A\ref{AS4}),(A\ref{AS5}) in Lemma \ref{lem4}. Notice that (A\ref{AS5}) holds thanks to $(p+q-1)s < 1$. We change (A\ref{AS2}) to the condition that $f$ analytically continues to $\C^-\cup(0,1)\cup \C^+$ and $f|_{\C^-}$ extends to a continuous function on $\C^-\cup(-\infty,0)$. The analytic continuation is given by just replacing $x$ with $z$ in \eqref{eqii}.

In order to show (A\ref{AS7}) in Lemma \ref{lem4}, we compute
\be
f(x-\I0) = -\frac{s}{B(p,q)} \cdot |x|^{ps-1}e^{-\I\pi p s} (1-|x|^se^{-\I \pi s})^{q-1},\qquad x<0. 
\ee 
Since $\arg(1 - |x|^s e^{-\I\pi s}) \in (0,\pi-\pi s)$, we have that $\arg(-f(x-\I0)) \in (-\pi p s,\pi(1-p- q)s + \pi(q-1))$. Our assumptions $0< p s < \frac{1}{2}, q-\frac{3}{2} \leq s(p+q-1)$ imply that $\arg(-f(x-\I0)) \in (-\frac{\pi}{2}, \frac{\pi}{2})$ and hence assumption (A\ref{AS7}) holds true. So the curve $g_3$ lies on $\C^+\cup\R$. 

The asymptotics \eqref{eq cauchy near 0} holds in the present case too (the constant $a$ may change and $p$ is replaced by $p s$), and hence for each $\epsilon>0$ there exists $\delta>0$ small enough so that 
%\be
%|G_\mu(\delta e^{\I\theta})+a'(-\delta e^{\I\theta})^{p s -1}| <  \frac{1}{2}a'\delta^{p s-1}
%\ee
%uniformly on $\theta \in (-\pi,\pi)$. 
the distance between $g_2$ and 0 is larger than $\epsilon^{-1}$. The maximal argument of the curve $g_2$ is in $(0,\pi)$ as discussed in the proof of Lemma \ref{lem4}.

The above arguments and the proof of Lemma \ref{lem4} show that the curves $g_k,k=1,2,3$ typically behave as in Fig.\ \ref{fig 2}.
\begin{figure}[htpb]
\begin{minipage}{0.5\hsize}
\begin{center}
\includegraphics[width=50mm,clip]{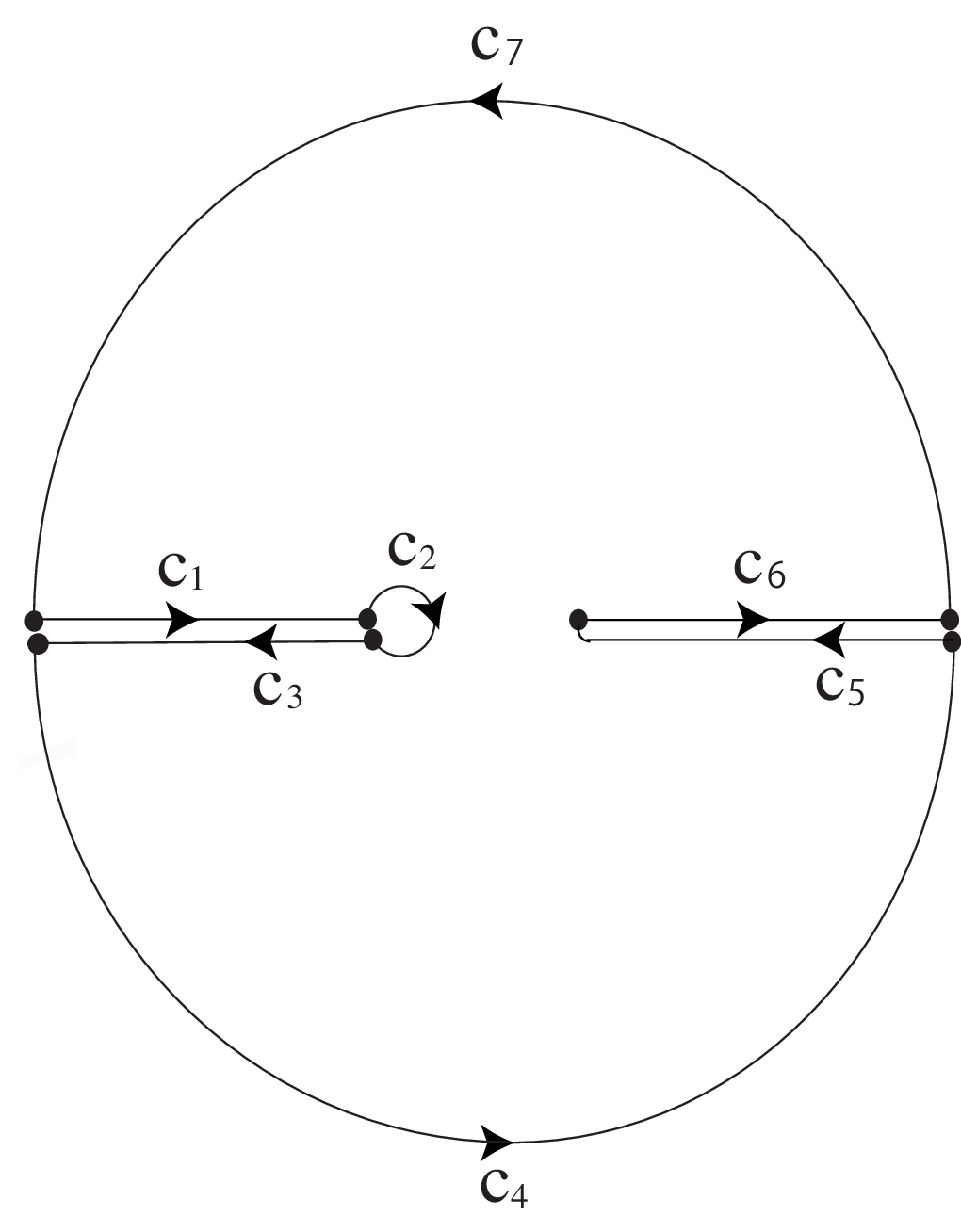}
\end{center}
  \end{minipage}
\begin{minipage}{0.5\hsize}
\begin{center}
\includegraphics[width=50mm,clip]{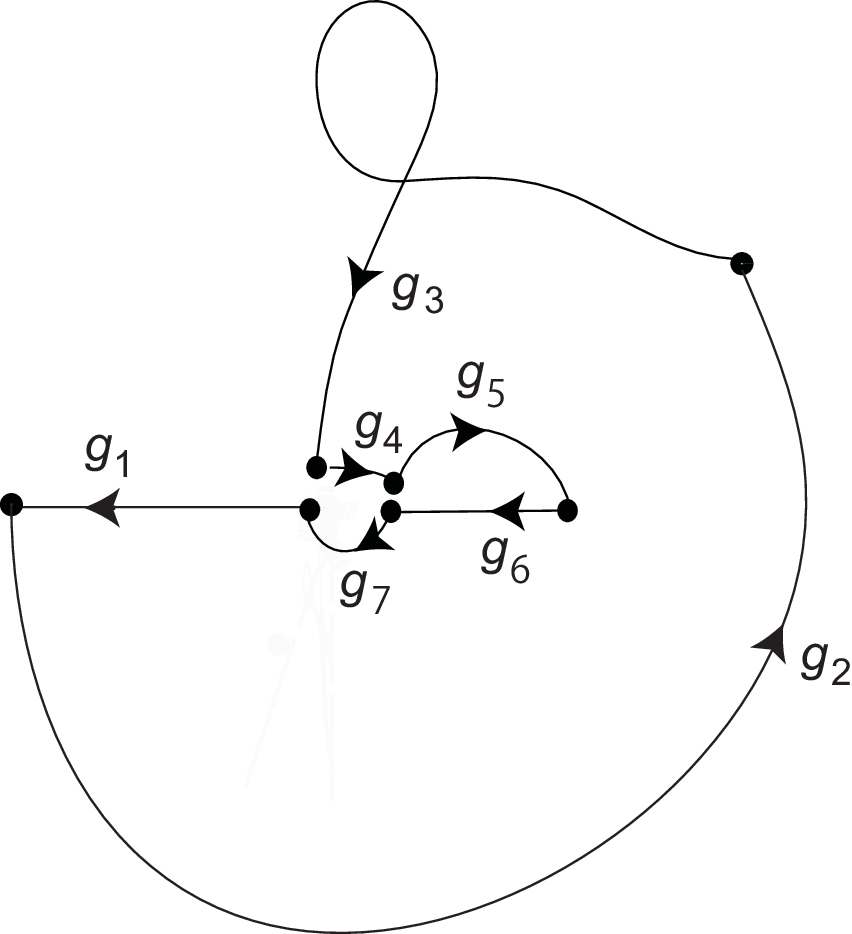}
\end{center}
\end{minipage}
\caption{The curves $c_k$ and $g_k$}\label{fig 2}
\end{figure}

\noindent
{\bf Step 2: Analysis of $G_{X^r}$ around $[1,\infty)$ and around $\infty$}. We show the following properties: 
\begin{enumerate}[\rm(i)]
\item\label{ASii} The restriction $f|_{\C^-}$ extends to a continuous function on $\C^- \cup \R \setminus\{0\}$; 
\item\label{ASi}  The analytic function $G:=G_{X^r}$ (defined in $\C^+\cup(0,1)\cup\C^-$ via Lemma \ref{lem analytic cauchy}) extends to a continuous function on $\C^+\cup(0,1)\cup\C^- \cup ([1,\infty)+\I0)\cup ([1,\infty)-\I0)$; 
\item\label{ASiii} $\re(f(x-\I0)) \leq0$ for $x>1$. 
\end{enumerate}
Note that in \eqref{ASi} we understand that $[1,\infty)+\I0, [1,\infty)-\I0$ are different half-lines by cutting the plane along $[1,\infty)$ and going to the Riemannian surface. This property implies that the map $(r,\theta)\mapsto G(re^{\I\theta})$ defined in $(0,1)\times(0,2\pi)$ extends to a continuous map on $[0,1)\times [0,2\pi]$, and in particular that $G(1+\I0)=G(1-\I0)$.

Property \eqref{ASii} is easy to show. 

Property \eqref{ASi} follows from 
the second asymptotics in \cite[(5.6)]{Has14}: there exists $b \in \R$ such that
\be
G(z) = b + o(1),\qquad z\to 1, z\in \C\setminus[1,\infty). 
\ee
The number $b$ is positive since it equals $\lim_{x\downarrow1} G(x+\I0)>0$. 
Note that the proof of this asymptotics required $1<q<2$ (in \cite{Has14} $q$ is denoted by $\alpha$), but we can give a proof for $2<q<3$ too only by using the identity $w^{\alpha-2}= (w-z)w^{\alpha-3} + z w^{\alpha-3}$ in \cite[(5.13)]{Has14}. For $q=2$ a logarithm term  appears and so we avoid such a case for simplicity.  The continuity on $(1,\infty)-\I0$ follows from formula \eqref{eq analytic cauchy} and property \eqref{ASii}.

Property \eqref{ASiii} follows from the computation
\be
f(x-\I0) = \frac{s}{B(p,q)} \cdot x^{ps-1} (x^s -1)^{q-1} e^{\I \pi (q-1)},\qquad x>1
\ee 
and our assumption $q \in [\frac{3}{2},\frac{5}{2}]$.

We moreover define $c_4,\dots,c_7$ and the corresponding images $g_k=G(c_k)$ for $k=4,\dots, 7$: 
\begin{itemize}
\item $c_4$ is a counterclockwise curve which lies on $\{z\in\C^-\mid |z|\geq\eta\}$, starting from $-\eta -\I0$ (the final point of $c_3$) and ends at $1+\eta-\I0$; 
\item $c_5$ is the line segment from $1+\eta -\I0$ to $1-\I0$; 
\item $c_6$ is the line segment from $1+\I0$ to $1+\eta+\I0$; 
\item $c_7$ is a counterclockwise curve which lies on $\{z \in\C^+\mid |z|\geq\eta\}$, starting from $1+\eta+\I0$ and ending at $-\eta+\I0$. 
\end{itemize}
Thanks to property \eqref{ASi},  $g_5 \cup g_6$ is a continuous curve, so one need not take a small circle to avoid the point 1. 

We can take a large $\eta>0$ similarly to Lemma \ref{lem4} so that 
the curves $g_4, g_7$ lie on the ball $\{z\in\C\mid |z|<\epsilon\}$ as shown in the right figure. 
This is easy to prove for $g_7$ since the measure is compactly supported and so $G(z)=O(1/z)~(z\to\infty, z\in\C^+)$. For $g_4$ we need Lemma \ref{lem analytic cauchy} and our assumption $s(p+q-1)<1$. 
Thanks to property \eqref{ASiii}, we have that 
\be
\im(G(z)) =  \im(\widetilde{G}(z)) - 2\pi\, \im(\I f(z)) \geq0
\ee
on $c_5$, so the curve $g_5$ is on $\C^+\cup\R$.

From the above arguments, every point of $D_\epsilon=\{w \in\C^-\mid  \epsilon< |w| < \epsilon^{-1}\}$ is surrounded by the closed curve $g_1 \cup\cdots  \cup g_7$ exactly once. So we can define a univalent inverse $G^{-1}$ in $\C^-$ as discussed in Lemma \ref{lem4}, to conclude that $X^r \sim\UI$. Approximation shows that the result is true for $r=1$ and $q=2$ and for the case $p+q-1=r$ too. From Lemma \ref{lem3} the law of $X^r$ is free regular. 
\end{proof}
\begin{rem} The curves $g_1$ and $g_6$ are Jordan curves. 
If $\eta>0$ is large and $\delta>0$ is small then we can moreover show that $g_2,g_4,g_7$ are also Jordan curves (with e.g.\ Rouche's theorem). However these properties are not needed to prove the theorem. 
\end{rem}

We will prove the results for powers of free Poisson rvs. The idea is again similar to the previous proofs. A new phenomenon is that the Cauchy transform has a singularity when we take the limit $z\to0, z\in\C^-$, and we will see how the previous methods are modified. 

\begin{proof}[Proof of Theorem \ref{thm3}]
Since approximation is allowed, we only consider parameters $(p,r)$ from an open dense subset of the full set.   

Case (a): $p>1,r>1$.  

(a-UI) We will first show that $X^r\sim \UI$ when $X\sim \MP(p,1)$. For simplicity we define $G(z)=G_{X^r}(z)$, $s=1/r$, $a = |\sqrt{p}-1|^{2/s}>0$ and $b=(\sqrt{p}+1)^{2/s}>a$. Then the pdf of $X^r$ is equal to 
\be\label{pdf21} 
f(x) = \frac{s }{2\pi} \cdot \frac{\sqrt{(b^s-x^s)(x^s-a^s)}}{x} \,1_{(a,b)}(x). 
\ee 
This density extends to an analytic function in $\C^+\cup(a,b)\cup \C^-$ by replacing $x$ in \eqref{pdf21} with a complex variable $z$. The function $z^s$ is the principal value. One can take the square root also as the principal value, but it may not be obvious. To justify this claim, we use the identity 
\be\label{identity}
(b^s-z^s)(z^s-a^s)=-\left(z^s-\frac{a^s+b^s}{2}\right)^2 +\left(\frac{b^s-a^s}{2}\right)^2.
\ee
 If $z \in\C^\pm$ then $z^s-\frac{a^s+b^s}{2} \in\C^\pm$, and hence $-(z^s-\frac{a^s+b^s}{2})^2 \in \C\setminus(-\infty,0]$, and hence by adding the positive real number $(\frac{b^s-a^s}{2})^2$, we conclude that $(b^s-z^s)(z^s-a^s) \in \C\setminus(-\infty,0]$. So the principal value is relevant for defining the square root.

The Cauchy transform $G$ extends to $\C^+\cup (a,b)\cup \C^-$ analytically via Lemma \ref{lem analytic cauchy}. The function $f|_{\C^-}$ extends to a continuous function on $\C^-\cup \R \setminus\{0\}$ since $(b^s-z^s)(z^s-a^s)$ continuously extends to $\C^-\cup \R$ without taking the value $0$ except at $a,b$. 
 % and then $(b^s-z^s)(z^s-a^s) \in \C^\pm$ for $z\in\C^\pm$ respectively. 

Curves $c_k, g_k=G(c_k)$ that we use in the proof are shown in Fig.\ \ref{fig 3}.

\begin{figure}[htpb]
\begin{minipage}{0.5\hsize}
\begin{center}
\includegraphics[width=50mm,clip]{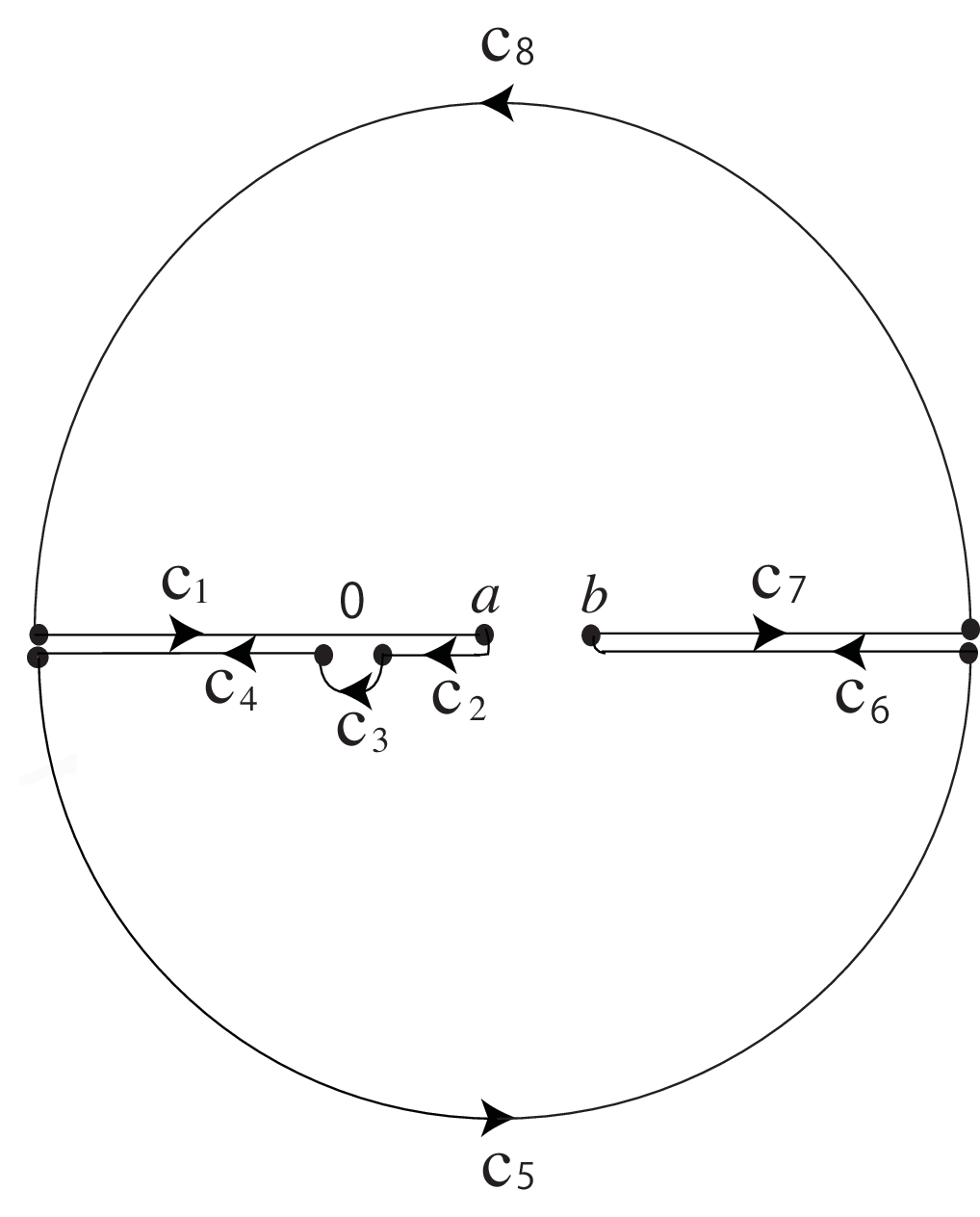}
\end{center}
  \end{minipage}
\begin{minipage}{0.5\hsize}
\begin{center}
\includegraphics[width=60mm,clip]{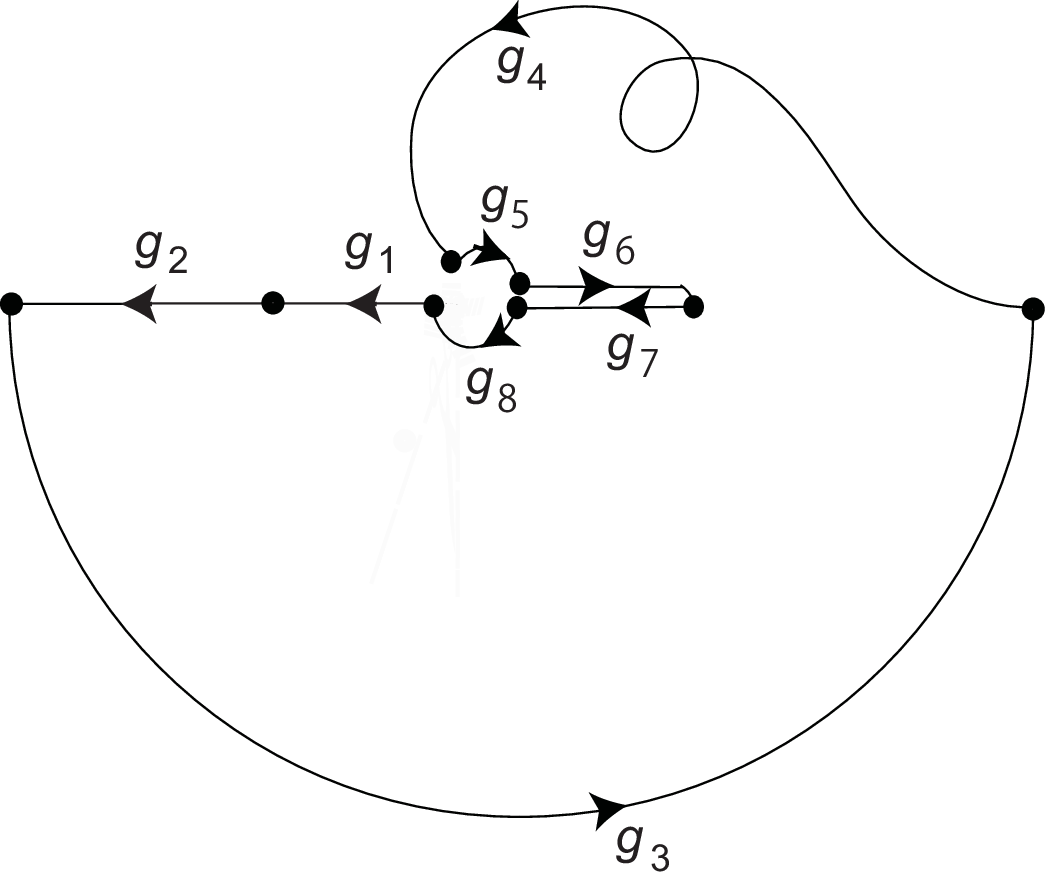}
\end{center}
\end{minipage}
\caption{The curves $c_k$ and $g_k$ $(p>1,r>1)$}\label{fig 3}
\end{figure}

We compute the boundary value $f(x-\I0)$ for $x<a, x>b$, which is the most crucial part of the proof. For $0<x<a$, since $(x-\I\delta)^s-a^s$ approaches the point $-(a^s- x^s) <0$ from $\C^-$ as $\delta \downarrow0$, we should understand that $(x-\I0)^s-a^s = (a^s-x^s) e^{-\I\pi +\I0}$. Hence 
\be\label{eq G2}
f(x-\I0) = \frac{s}{2\pi}\cdot \frac{e^{-\frac{\I\pi }{2}}\sqrt{(b^s-x^s)(a^s-x^s)} }{x},\qquad 0<x<a,  
\ee
and hence $\re(f(x-\I0))=0$. So the curve $g_2$ is on the (negative) real line.  

For $x>b$, since $b^s-(x-\I\delta)^s$ approaches a point in $(-\infty,0)$ from $\C^+$, we have to understand that $b^s-(x-\I0)^s = (x^s -b^s)e^{\I\pi-\I0}$. Hence 
\be\label{eq G3}
f(x-\I0) = \frac{s}{2\pi}\cdot \frac{e^{\frac{\I\pi }{2}}\sqrt{(x^s-b^s)(x^s-a^s)} }{x},\qquad x>b,   
\ee
and hence $\re(f(x-\I0))=0$. So the curve $g_6$ is on the (positive) real line. 

As we saw in \eqref{identity}, $\arg((b^s-z^s)(z^s-a^s))\in(-\pi,\pi)$, and hence $\sqrt{(b^s-z^s)(z^s-a^s)}\in \{w \mid \re(w)>0\}$. Dividing it by $z$ and taking the limit $z\to x-\I0=|x|e^{-\I\pi}$, we have 
\be\label{eq G4}
\re(f(x-\I0)) \leq 0,\qquad x<0  
\ee 
and so $g_4$ lies on $\C^+\cup\R$. 
The remaining proof is similar to Theorems \ref{thm1}, \ref{thm2}, so we only mention important remarks. The proof of \cite[Theorem 5.1(5.6)]{Has14} for $\alpha=\frac{3}{2}, x_0=a$ and $x_0=b$ (with reflection) enables us to show that the Cauchy transform $G$ has a continuous extension to 
\be
\C^+ \cup [a,b] \cup \C^- \cup((-\infty,a]\cup [b,\infty)+\I0) \cup ((-\infty,a] \cup [b,\infty)-\I0) \setminus\{0-\I0\}, 
\ee
similarly to property \eqref{ASi} in the proof of Theorem \ref{thm2}. 
This implies that $g_1\cup g_2$ is a continuous curve and so is $g_6\cup g_7$. The Cauchy transform has a singularity at $0$ when approaching from $\C^-$. This singularity is a contribution of $x^{-1}$ of (the analytic continuation of) the pdf $f$. So we avoid 0 via the curve $c_3$. When one draws the picture of $g_3$, one should take it into account that $G(z) = -2\pi \I f(z) +o(|z|^{-1}) = -\frac{s(p-1)}{z}(1+o(1)),$ $z\to0, z\in\C^-$ since $\widetilde{G}(z)$ is analytic around $0$ and $f(z) =\frac{s}{2\pi}\cdot \frac{-\I \sqrt{a^s b^s}}{z}(1+o(1))$. So the curve $g_3$ is as shown in Fig.\ \ref{fig 3}.  Thus one can show that $X^r \sim \UI$.\footnote{In the published version there was case (a-FR), where it is stated that "We can see that ... for any $t\in (0,a)$ there exists $\delta_2>0$ such that  $G(a-t-\I y) \in\C^-$ for $y\in(0,\delta_2)$ thanks to \eqref{eq G2}", which is not clear. The statement may be true, but at least the argument is not satisfactory or suitable. As a consequence it is not clear whether $X^r$ is FR or not.}

%(a-FR)  Since we cannot apply Lemma \ref{lem3}, we will directly prove that $X^r \sim \FR$. Let us take the curves $c_1= (t+\I0)_{t\in(-\infty,a]}$ and $c_2=(a-t-\I0)_{t\in[0,a]}$. We can see that for any $t \in(-\infty,a)$ there exists $\delta_1 >0$ such that $G(t+\I y)\in\C^-$ for all $y\in(0,\delta_1)$ (since $G$ maps $\C^+$ to $\C^-$) and also for any $t\in (0,a)$ there exists $\delta_2>0$ such that  $G(a-t-\I y) \in\C^-$ for $y\in(0,\delta_2)$ thanks to \eqref{eq G2}. Note also that $G(c_1+\I0)=[-\alpha, 0)$ and $G(c_2-\I0)=(-\infty,-\alpha]$ for $\alpha:= \widetilde{G}(a-0)>0$. These facts imply that our inverse $G^{-1}$ in $\C^-$ satisfies that 
%\be\label{eq aa}
%G^{-1}((-\infty,-\alpha]-\I0)=c_2,\qquad G^{-1}([-\alpha,0)-\I0)=c_1. 
%\ee
%It also holds that $G^{-1}(-\infty-\I0)=0$ since $G(z)$ tends to $\infty$ as $z\to0, z\in\C^-$. Let $\varphi(z):=G^{-1}(z^{-1}) -z, z\in\C^+$ be the Voiculescu transform. Since $\im(\varphi(x+\I0)) =0,x<0$ by \eqref{eq aa}, the free L\'evy measure of $X^r$ has no mass on $(-\infty,0)$ by Stieltjes inversion.  Since $\varphi(-0)=G^{-1}(-\infty)=0$, \cite[Proposition 11]{AHS13} implies that $X^r \sim \FR$. 

Case (b): $p<1, r>1$. We follow the notations in case (a).  Since now the distribution of $X^r$ has an atom at 0 and hence we can apply Lemma \ref{lem3}, it suffices to show $X^r\sim \UI$. The probability distribution of $X^r$ is of the form $(1-p)\delta_0+f(x)\,dx$, where $f$ is given by the same formula \eqref{pdf21}. The functions $f$ and $G$ have analytic continuation as discussed in case (a). 

We will take curves $c_k$ as shown in Fig.\ \ref{fig 4}. 
\begin{figure}[htpb]
\begin{minipage}{0.5\hsize}
\begin{center}
\includegraphics[width=50mm,clip]{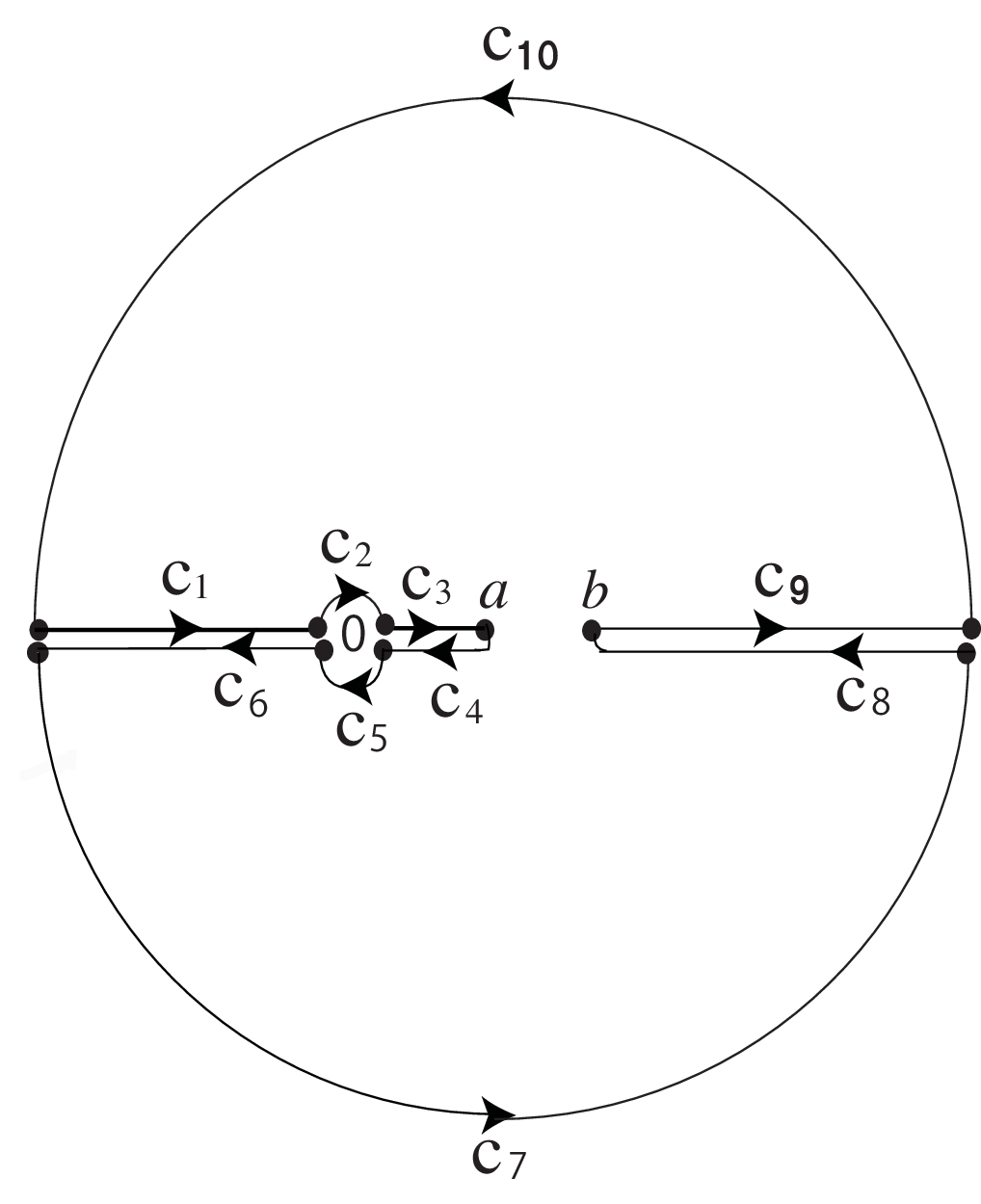}
\end{center}
  \end{minipage}
\begin{minipage}{0.5\hsize}
\begin{center}
\includegraphics[width=60mm,clip]{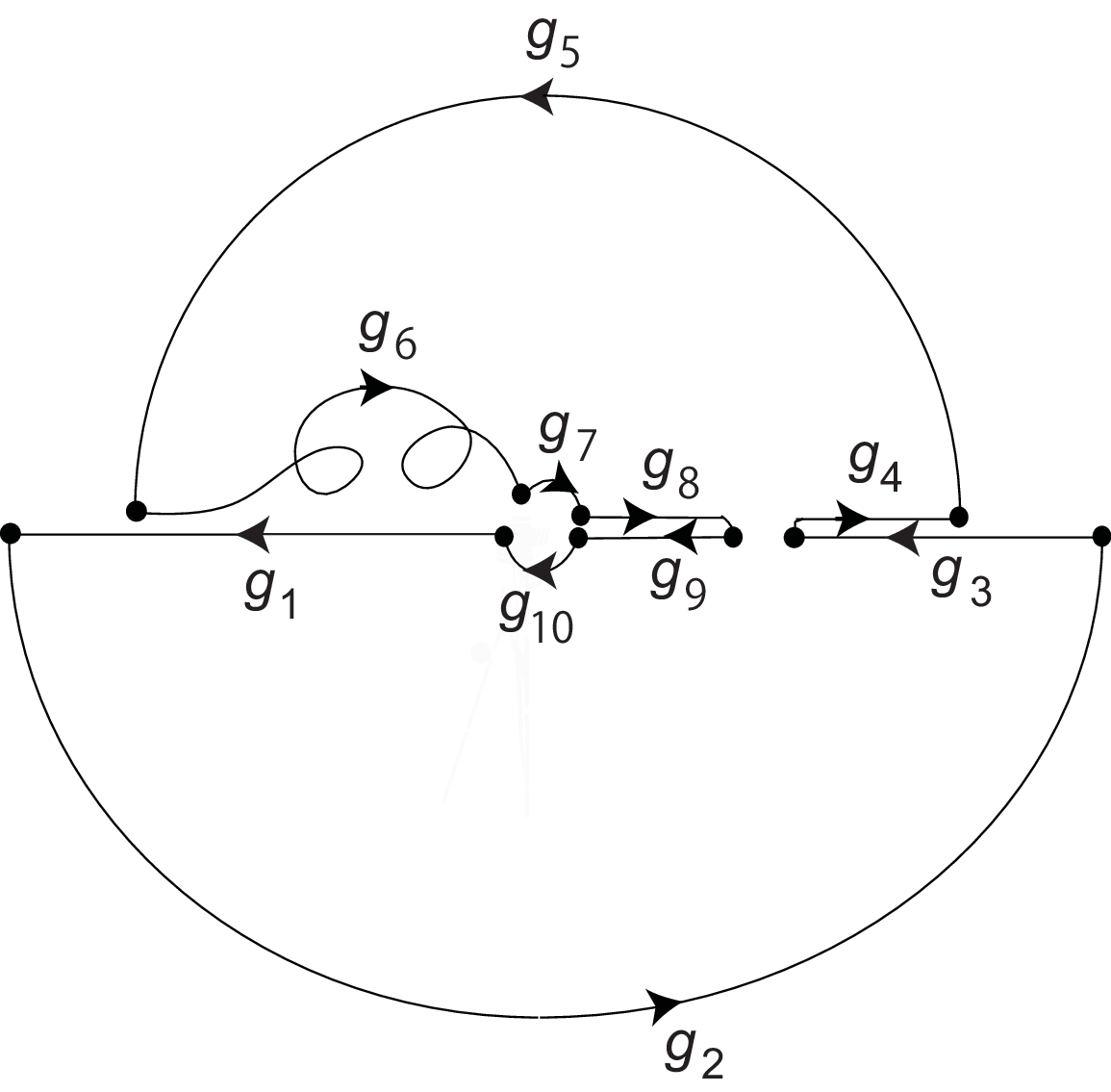}
\end{center}
\end{minipage}
\caption{The curves $c_k$ and $g_k$ $(p<1, r>1)$}\label{fig 4}
\end{figure}
Now the difference from case (a) is that $\widetilde{G}$ has a pole at $0$, and so we avoid $0\pm\I0$ via the curves $c_2,c_5$. 
Note that $G(z)=\frac{1-p}{z}(1+o(1))$ as $z\to0, z\in\C^+$ and $G(z)= \widetilde{G}(z) -2\pi\I f(z) = \frac{(1-p)(1-s)}{z}(1+o(1))$ as $z\to0, z\in\C^-$, so the curves $g_2, g_5$ look like large semicircles. Thus we can prove the claim $X^r\sim\UI$.

Case (c{}): $p>1, r<-1$. For simplicity we use the same notation $G(z)=G_{X^{r}}(z)$, but now we define 
$t=-1/r \in(0,1)$, $A = (\sqrt{p}+1)^{-2/t}>0$ and $B=(\sqrt{p}-1)^{-2/t}>A$. Then the pdf of $X^r$ is equal to 
\be\label{pdf22} 
h(x) = \frac{t (p-1)}{2\pi} \cdot \frac{\sqrt{(B^t-x^t)(x^t-A^t)}}{x^{t+1}} \,1_{(A,B)}(x). 
\ee 
The functions $h, G$ continue analytically to  $\C^+\cup(A,B)\cup \C^-$ and $h|_{\C^-}$ extends to a continuous function on $\C^-\cup\R\setminus\{0\}$ as we discussed in case (a). 

We will take the same curves $c_k$ as in case (a) except that the points $a,b$ are replaced by $A,B$ respectively (see Fig.\ \ref{fig 5}). 
\begin{figure}[htpb]
\begin{minipage}{0.5\hsize}
\begin{center}
\includegraphics[width=50mm,clip]{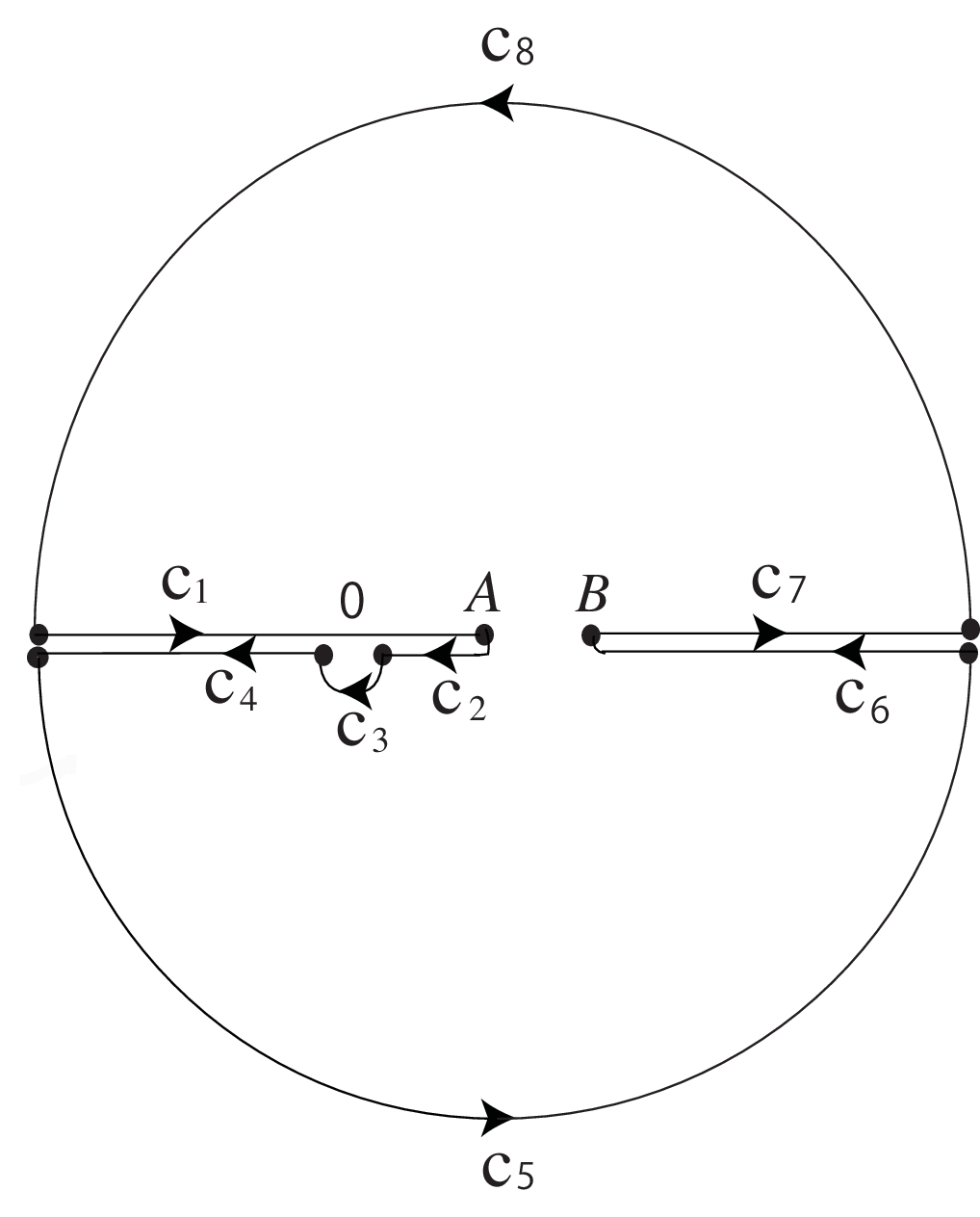}
\end{center}
  \end{minipage}
\begin{minipage}{0.5\hsize}
\begin{center}
\includegraphics[width=55mm,clip]{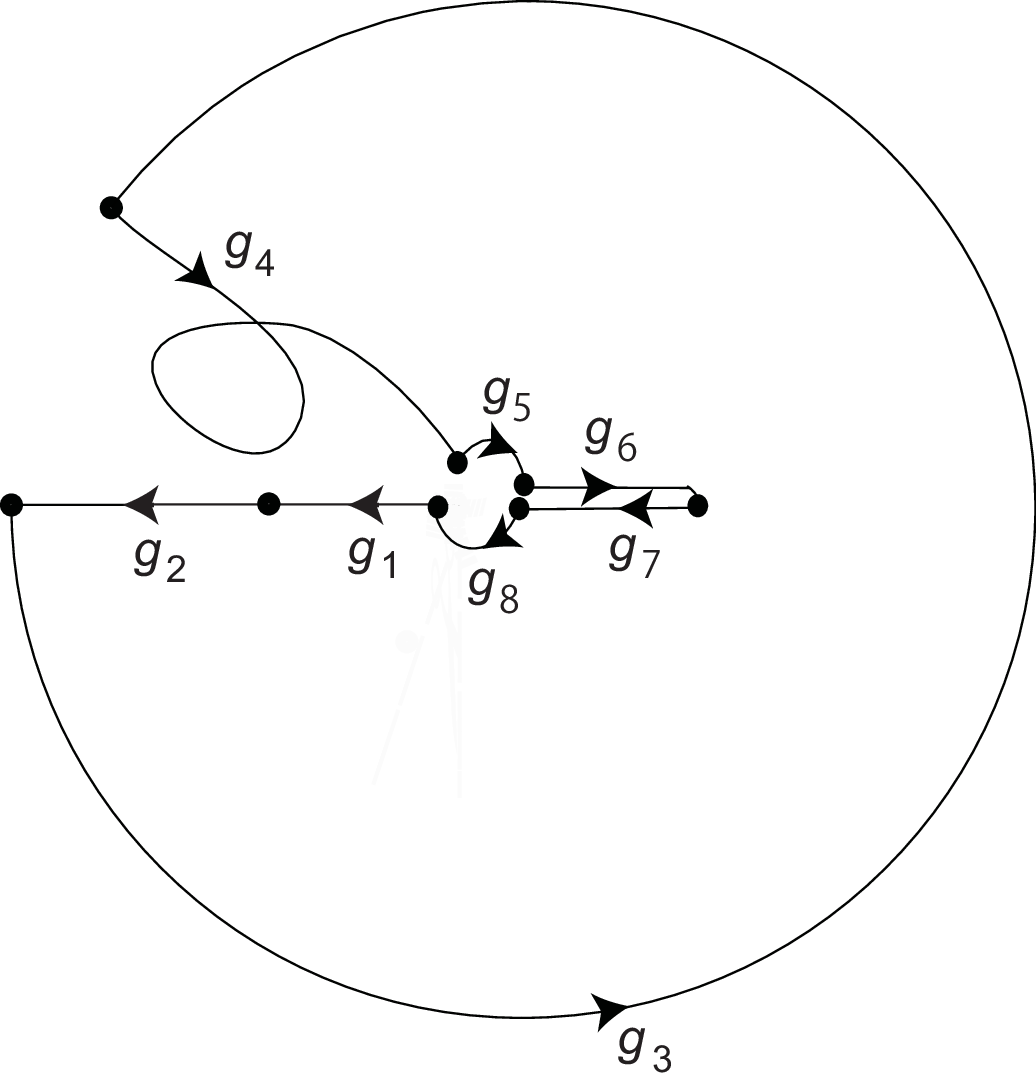}
\end{center}
\end{minipage}
\caption{The curves $c_k$ and $g_k$ $(p>1,r<-1)$}\label{fig 5}
\end{figure}

In the present case, the difficulty is to show the inequality corresponding to \eqref{eq G4}, i.e.\ $\re(h(x-\I0)) \leq 0$ for $x<0$. The inequality $\re(h(x-\I0)) \leq 0$ (actually equality holds) for $0<x<a$ and $x>b$ is easy because only the new factor $(p-1) x^{-t} >0$ is different from  \eqref{eq G2},\eqref{eq G3}.  

We show that $\re(h(x-\I0)) \leq 0$ for $x<0$ by considering two further sub cases. 

Case (ci): $0<t \leq\frac{1}{2}$. We will use the identity \eqref{identity} with $(a, b)$ replaced by $(A,B)$. First we can easily show that $\arg_{(-2\pi,0)}(((x-\I0)^t-(A^t+B^t)/2)^2)\in(-2\pi, -2\pi t)$ for $x<0$. Then we multiply it by $-1$ and shift it by $((B^t-A^t)/2)^2$ and so $\arg((B^t-(x-\I0)^t)((x-\I0)^t-A^t))\in(-\pi, (1-2 t)\pi)$. 
Then we take the square root to get $\arg(\sqrt{(B^t-(x-\I0)^t)((x-\I0)^t-A^t)})\in(-\frac{\pi}{2}, (\frac{1}{2}-t)\pi)$. Further we divide it by $(x-\I0)^{t+1} =|x|^{t+1}e^{-\I\pi(t+1)}$, to get $\re(h(x-\I0)) \leq 0$ for $x<0$. 

Case (cii): $\frac{1}{2}<t<1$. We can show that the curve $\gamma: ((x-\I0)^t-(A^t+B^t)/2)^2, x<0$ is on the right side of the half-line $L: \left((A^t+B^t)/2\right)^2 + r e^{\I(2\pi-2\pi t)}, r>0$, and then $\arg((B^t-(x-\I0)^t)((x-\I0)^t-A^t) )\in(-\pi, (1-2 t)\pi)$ (see Remark \ref{detail} for more details). The remaining arguments are the same as case (ci). 

Thus we have now $\re(h(x-\I0)) \leq 0$ for $x<0$ and hence $g_4$ lies on $\C^+\cup\R$. The remaining proof is almost the same as in case (a), but there are two remarks. One is that as $z\to0, z\in \C^-$, the Cauchy transform $G(z)=\widetilde{G}(z)-2\pi\I h(z)$ has asymptotics $G(z) = -\frac{t}{z^{t+1}}(1+o(1))$ which is different from case (a) by the factor $z^{-t}$ and a constant. Therefore the arguments of $g_3$ change from around $-\pi$ to $t\pi$ as in Fig.\ \ref{fig 5}. The other remark is that as $z\to\infty, z\in \C^-$, the asymptotics of the Cauchy transform is $G(z)=\widetilde{G}(z)-2\pi\I h(z) = (1-t(p-1))\cdot z^{-1}(1+o(1))$, so $g_5$ is near 0. Thus we are able to show that $X^r \sim \UI$. The proof of $X^r \sim\FR$ is similar to case (a).

Case (d): $p>1, r\in(-1,0)$. The pdf is the same as \eqref{pdf22}. We follow the notations in case ({}c), and so now we have $t>1$. The function $h$ now analytically continues to $\{z\in\C\setminus\{0\}\mid \arg(z)\in(-\pi/t,0) \cup(0,\pi/t)\}\cup(A,B)$ and so $G$ extends to $\C^+\cup(A,B)\cup\{z\in\C\setminus\{0\}\mid \arg(z)\in(-\pi/t,0)\}$. Moreover $h$ extends to a continuous function on $\{z\in\C\setminus\{0\}\mid \arg(z)\in[-\pi/t,0]\}$. The curve $c_4$ in case (c{}) is not useful now, and instead we take $c_4$ to be a long line segment with angle $-\pi/t$. 
So we take curves as shown in Fig.\ \ref{fig 6}. 
\begin{figure}[htpb]
\begin{minipage}{0.5\hsize}
\begin{center}
\includegraphics[width=50mm,clip]{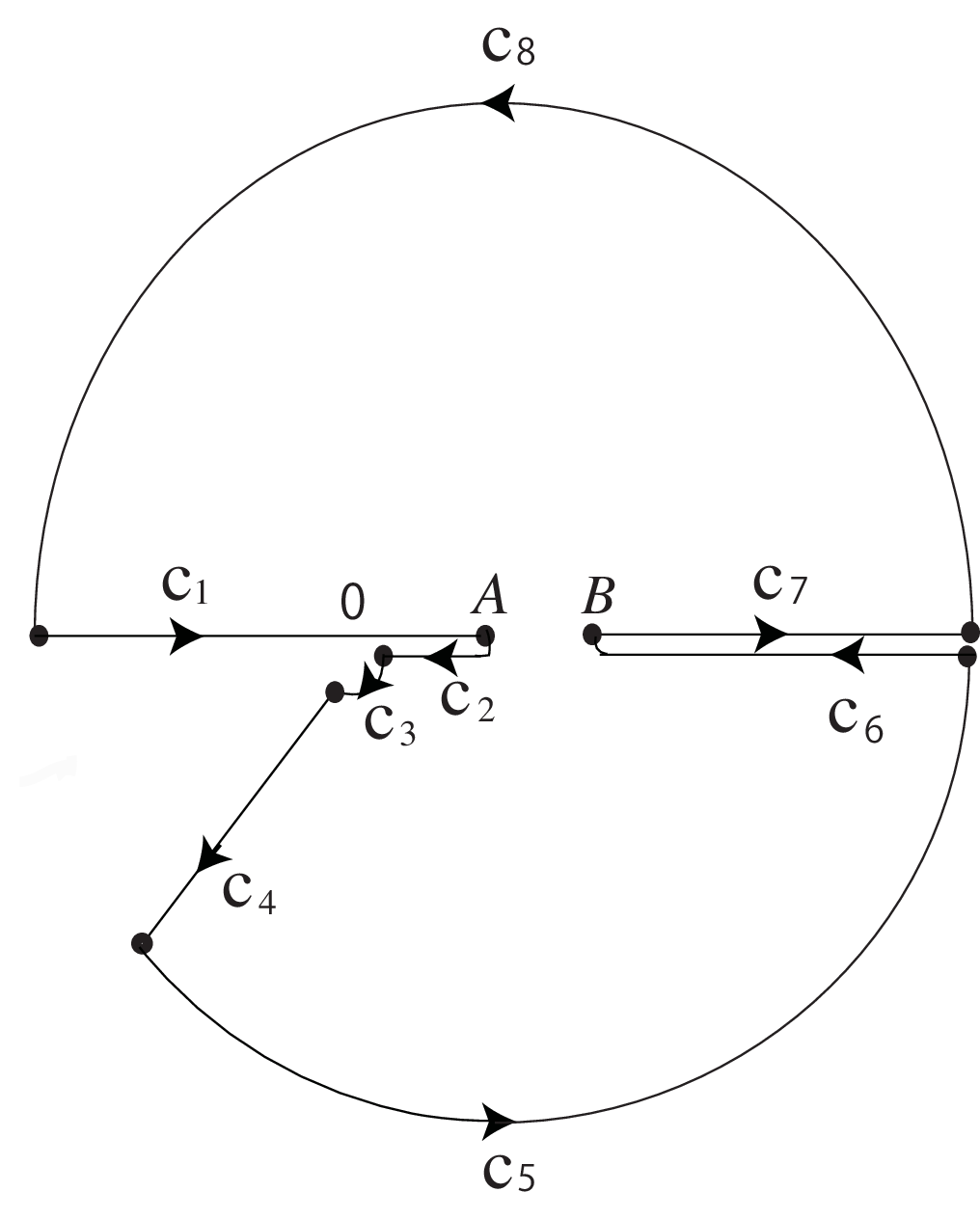}
\end{center}
  \end{minipage}
\begin{minipage}{0.5\hsize}
\begin{center}
\includegraphics[width=55mm,clip]{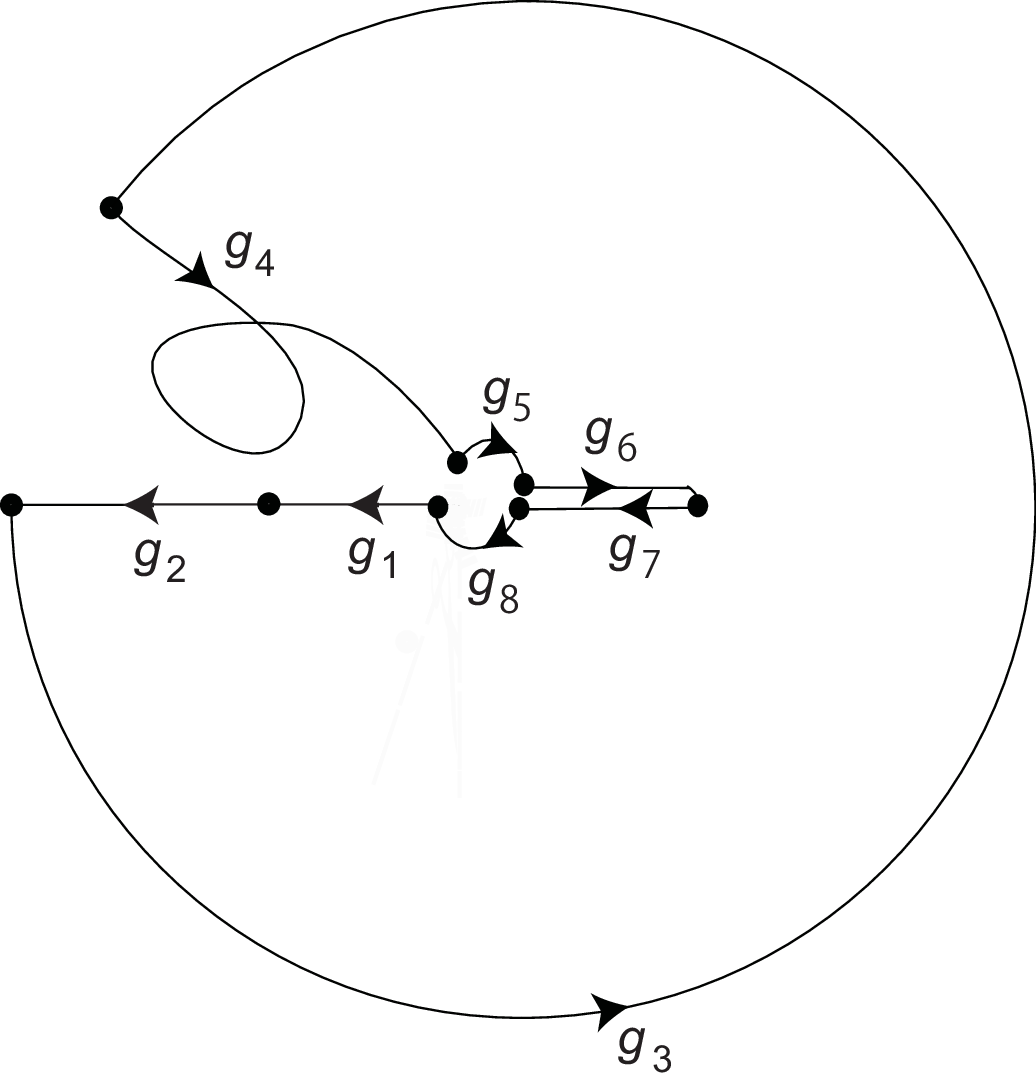}
\end{center}
\end{minipage}
\caption{The curves $c_k$ and $g_k$ $(p>1, -1<r<0)$}\label{fig 6}
\end{figure}

We compute the boundary value on $c_4$: 
\be
\begin{split}
\lim_{\delta\downarrow0}h(r e^{-\I\pi/t}+\delta) 
&= \frac{t (p-1)}{2\pi} \cdot \frac{e^{-\I\pi/2}\sqrt{(B^t+r^t)(A^t+r^t)}}{r^{t+1} e^{-\I\pi(t+1)/t}} \\
&=  \frac{t (p-1)\sqrt{(B^t+r^t)(A^t+r^t)} }{2\pi r^{t+1}} \cdot \I e^{\I\pi/t}, \qquad r>0,   
\end{split}
\ee
and hence $\re(h(r e^{-\I\pi/t}+0))<0$ as desired. This implies that $g_4$ lies on $\C^+\cup\R$. The remaining proof is similar to case (c{}) and the typical behavior of the curves $g_k$ is also similar to case (c{}). Thus we can show that $X^r \sim \UI$. %The proof of $X^r\sim\FR$ is similar to case (a). 
\end{proof}

\begin{rem}[More details on case (cii)]\label{detail} For simplicity let  $d$ denote $(A^t +B^t)/2$. 
We first show that the curve $\gamma({}r)= \ell({}r)^2= (r e^{-\I t \pi}-d)^2, r>0$ is on the right side of the half line $L({}r)= d^2 + r e^{\I(2\pi-2\pi t)}, r>0$ (see Figs \ref{fig a},\ref{fig b}).  Our claim is equivalent to $\arg(\gamma(r{})-d^2)\in(0, 2\pi-2t\pi), r>0$. From Fig.\ \ref{fig a} and the current assumption $1/2 < t <1$, we can see that $\arg(\gamma({r})-d^2)\in(0,\pi)$. A simple computation shows that $\gamma(r{})-d^2 = (r e^{-\I t \pi}-d)^2 -d^2 = r^2 \cos(2 t\pi) -2d r \cos(t\pi) -2\I r \sin(t\pi)(r \cos(t\pi)-d)$. After some more computations, one has 
\be
\im( (\gamma({r})-d^2)e^{-\I(2\pi-2t\pi)}) = -2 d r \sin(t\pi) <0.  
\ee
Therefore $\arg(\gamma({r})-d^2)< 2\pi-2t\pi$. 

Next, since $d^2=(\frac{A^t+B^t}{2})^2>(\frac{B^t-A^t}{2})^2$, the arguments of the curve $\gamma -(\frac{B^t-A^t}{2})^2$ still lie in $(0, 2\pi-2t\pi)$. Hence the arguments of the curve $-\gamma +(\frac{B^t-A^t}{2})^2$ lie in $(-\pi, (1-2 t)\pi)$. 

\begin{figure}[htpb]
\begin{minipage}{0.5\hsize}
\begin{center}
\includegraphics[width=50mm,clip]{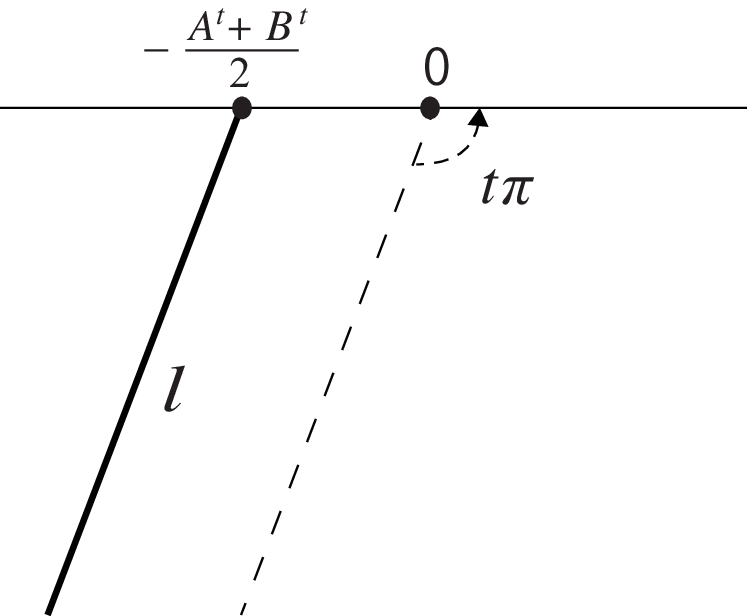}
\end{center}
  \hangcaption{$\ell: (x-\I0)^t-\frac{A^t +B^t}{2}, x<0$}\label{fig a}
  \end{minipage}
\begin{minipage}{0.5\hsize}
\begin{center}
\includegraphics[width=50mm,clip]{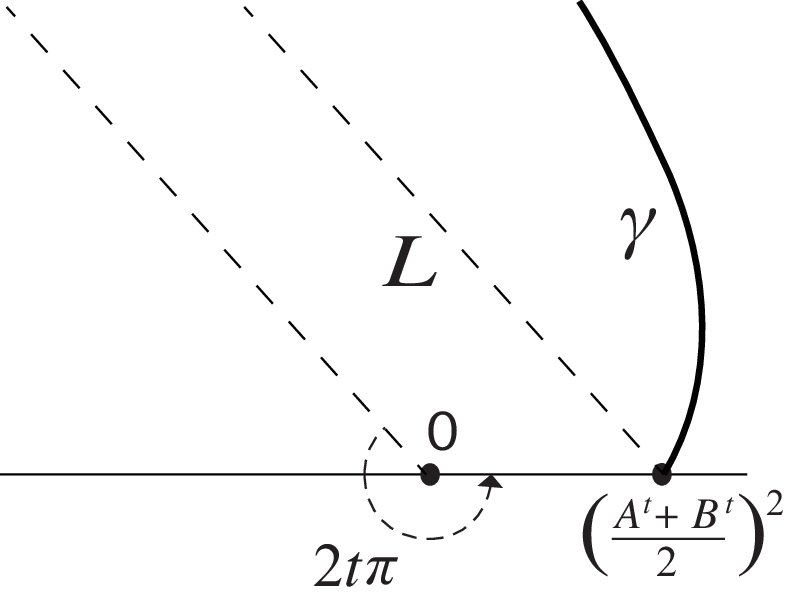}
\end{center}
\hangcaption{$\gamma: \left((x-\I0)^t-\frac{A^t +B^t}{2}\right)^2, x<0$}\label{fig b}
\end{minipage}
\end{figure}
\end{rem}

Finally we prove Theorem \ref{Sym Power Beta}, i.e.\ the FID properties of symmetric distributions. 
The negative result in Theorem \ref{Sym Power Beta}\eqref{Sym Beta not FID} on the symmetrized powers of beta rvs entails the following fact. 
 \begin{lem}\label{lem pdf 0}
Suppose that a probability measure $\mu$ on $\R$ is symmetric and $\mu|_{(-\epsilon,\epsilon)} = p(x)1_{(-\epsilon,\epsilon)}(x)\,dx$ for some $\epsilon>0$ and some continuous function $p$ on $(-\epsilon,\epsilon)$. 
 If $p(0)=0$, then $\mu \notin\FID$. 
 \end{lem}
 \begin{proof} Suppose that $\mu\in\FID$. 
The Cauchy transform of $\mu$ vanishes at 0 since $\re(G_\mu(\I y))=0$ and $\lim_{y\downarrow0}G_\mu(\I y)=-\I\pi p(0)=0$. This contradicts the fact that the reciprocal Cauchy transform $1/G_\mu(z)$ extends to a continuous function on $\C^+\cup\R$ into itself (see \cite[Proposition 2.8]{Bel05}). 
% or the proof of \cite[Proposition 5.5]{Has14}).  
 \end{proof}

\begin{proof}[Proof of Theorem \ref{Sym Power Beta}] 
The negative result \eqref{Sym Beta not FID} holds since if $p>r>0$ or $r\leq0$ then the law of $B X^r$ satisfies the assumptions of Lemma \ref{lem pdf 0}. 

We will prove the case \eqref{Sym Beta} and later give some comments on how to prove the HCM case (1). 
The idea is similar to the proof of Theorem \ref{thm2}, so we only mention significant changes of the proof. 
We follow the notations in the proof of Theorem \ref{thm2}. First we will check condition (b) in Definition \ref{UIS}. 
The pdf of $B X^r$ is given by 
\be
\wf(x) = \frac{s}{ 2 B(p,q)} \cdot |x|^{p s-1} (1-|x|^s)^{q-1}\,1_{(-1,1)\setminus\{0\}}(x).    
\ee 
The restriction $\wf|_{(0,1)}$ has the analytic continuation $\wf(z)=\frac{1}{2}f(z)$ in $\C^+\cup(0,1)\cup\C^-$ where $f(x)$ is the pdf of $X^r$. Note then that $\wf(x+\I0)\neq \wf(x)$ for $x\in(-1,0)$. Also $G_{B X^r}(z)= \frac{1}{2}(G_{X^r}(z)-G_{X^r}(-z))$ has analytic continuation to $\C^+\cup(0,1)\cup\C^-$ and 
\be
\begin{split}
G_{B X^r}(z) &= \frac{1}{2}\left(\widetilde{G}_{X^r}(z)-2\pi\I \wf(z)\right)-\frac{1}{2}\widetilde{G}_{X^r}(-z)\\
&= \widetilde{G}_{B X^r}(z)-\frac{1}{2}\pi\I f(z),\qquad z\in\C^-. 
\end{split}
\ee

We take curves $c_k$ in Fig.\ \ref{fig A} and consider the image curves $g_k=G_{B X^r}(c_k)$. 
\begin{figure}[htpb]
\begin{minipage}{0.5\hsize}
\begin{center}
\includegraphics[width=50mm,clip]{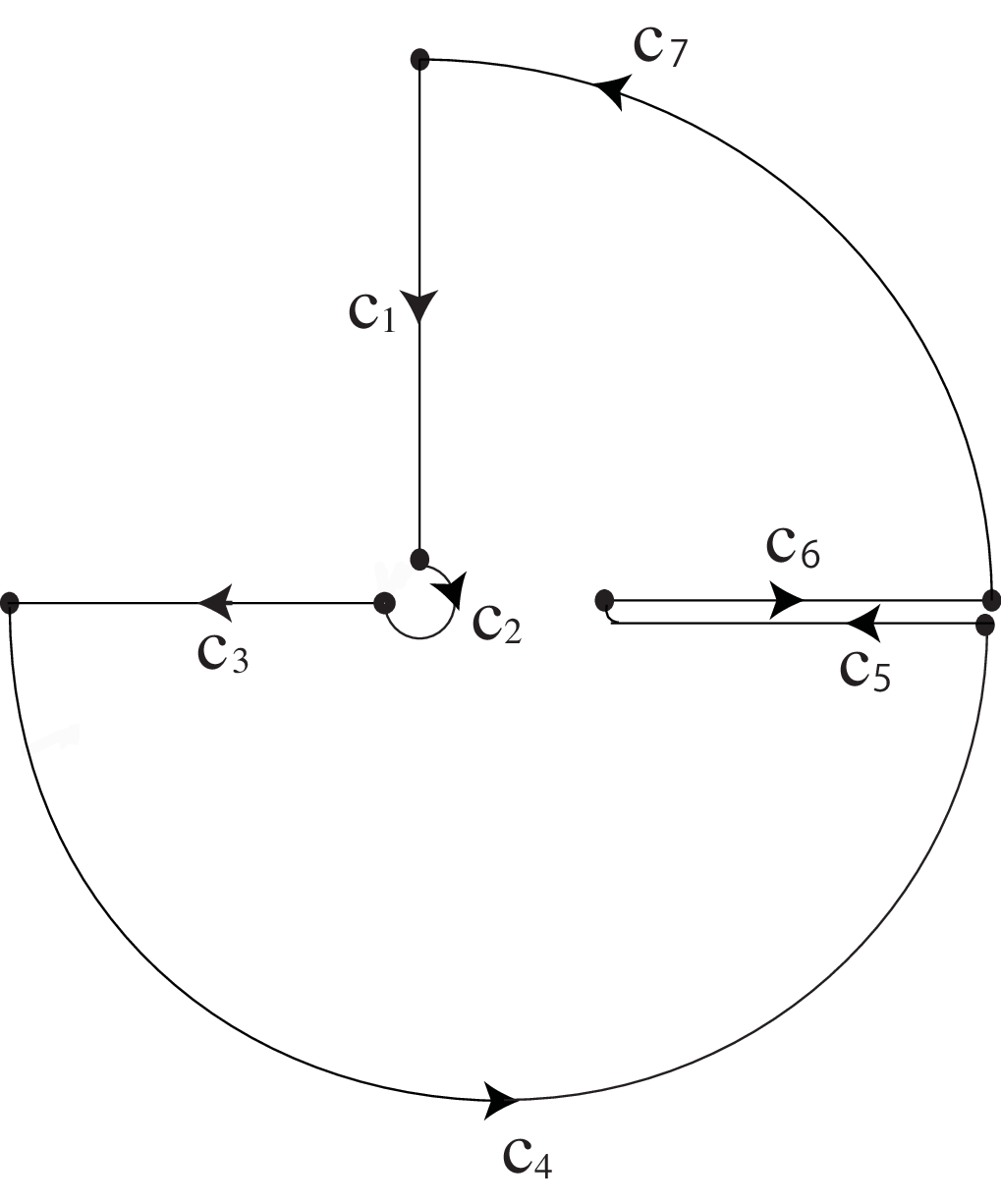}
\end{center}
  \end{minipage}
\begin{minipage}{0.5\hsize}
\begin{center}
\includegraphics[width=50mm,clip]{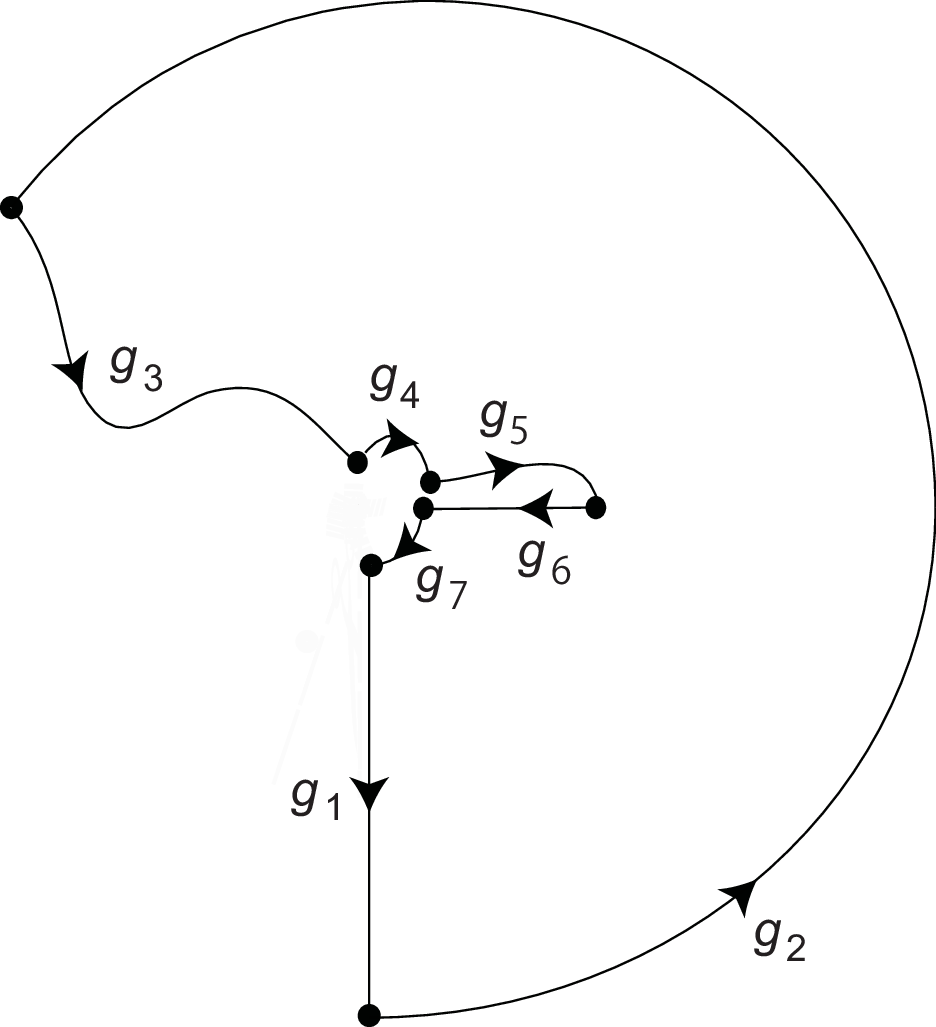}
\end{center}
\end{minipage}
\caption{The curves $c_k, g_k$}\label{fig A}
\end{figure}

We have proved  $\re(f(x-\I0)) \leq0$ for $x<0$ and $x>1$ in the proof of Theorem \ref{thm2}, so the curves $g_3,g_5$ are contained in $\C^+\cup\R$.  

The function $G_{B X^r}$ extends continuously to 
\be
\C^+\cup(0,1)\cup\C^- \cup([1,\infty)+\I0)\cup([1,\infty)-\I0)
\ee
as we proved in Theorem \ref{thm2}, so $g_5\cup g_6$ is a continuous curve.  

We may use the asymptotics \eqref{eq cauchy near 0} (we use the same notation $a$ for the constant) and obtain   
\be\label{eq009}
\begin{split}
G_{B X^r}(z) 
&= \frac{1}{2}(G_{X^r}(z)-G_{X^r}(-z)) \\
&=\frac{a}{2}(-(-z)^{p s-1}+z^{p s-1}) + o(|z|^{p s-1})\\
&= a \cos\left(\frac{\pi p s}{2}\right) e^{-\I \frac{\pi p s}{2}}z^{p s-1} +o(|z|^{p s-1})
\end{split}
\ee
as $z\to0, \arg(z) \in(-\pi,\pi)$. So when $z$ goes along a small circle centered at 0 from $\arg(z)=\pi/2$ to $\arg(z)=-\pi$, the Cauchy transform $G_{B X^r}(z) $ goes along a large circle-like curve from the angle $-\pi/2$ to $(1-\frac{3}{2} p s)\pi $ $(>0$  since $p s \leq1/2$) with small errors. This suggests that the curve $g_2$ seems as in Fig.\ \ref{fig A}.  

The asymptotics \eqref{eq009} also implies that $\frac{1}{\I} G_{B X^r}(\I y)\to -\infty$ as $y\downarrow0$. It is elementary to show that $G_{B X^r}(z)\to0$ as $|z|\to \infty, z\in\C^+$. Therefore the curves $g_1,g_7$ are as shown in Fig.\ \ref{fig A}. 

Thus each point of $\C^-\cap\I\C^-$ is surrounded by $g_1\cup\cdots \cup g_7$ exactly once when the curve $c_1 \cup\cdots\cup c_7$ is sufficiently close to the boundary of $\C^-\cup\I\C^-\setminus[1,\infty)$. Therefore $G_{B X^r}^{-1}$ has univalent analytic continuation to $\C^-\cap\I \C^-$ such that $G_{B X^r}^{-1}(\C^-\cap\I \C^-)\subset \C^-\cup\I \C^-$. This is condition (b) in Definition \ref{UIS}.

We will check condition (a) in Definition \ref{UIS}. By calculus we can show that $w'(x)<0$ for $x\in(0,1)$ and hence the distribution of $B X^r$ is unimodal. The map $G_{B X^r}$ is univalent in $\C^+$ thanks to \cite[Theorem 3]{Kap52} (see \cite[Theorem 39]{AA75} for a different proof). Recall that $\frac{1}{\I} G_{B X^r}(\I y)\to -\infty$ as $y\downarrow0$ and that $G_{B X^r}(\I y)\to0$ as $y\to \infty$. Therefore, $G_{B X^r}$ maps a neighborhood of $\I(0,\infty)$ onto a neighborhood of $\I(-\infty,0)$ bijectively. Thus we have condition (a) and therefore $B X^r \sim \UIS$.

For the HCM case, we may assume that $W$ takes only finitely many values. Then the pdf of $X$ is of the form \eqref{pdf HCM} and $r\geq1$. By calculus we can show that the pdf of $W X^r$ is unimodal. Then the proof is similar to the case of beta rvs.  
\end{proof}

According to Mathematica, the Cauchy transform of $X^r$ has an expression in terms of generalized hypergeometric functions when $X\sim \Beta(p,q)$ and at least $r=3/2,4$. Figs \ref{dom 1}--\ref{dom 4} show the numerical computation of the domain $G_{B X^{r}}^{-1}(\C^-\cap \I\C^-)$ which equals a connected component of $\{z\in \C^-\cup\I\C^-\setminus[1,\infty)\mid G_{B X^r}(z)\in\C^-\}$. 
These figures suggest that the law of $B X^r$ belongs to $\UI$ since the domain seems to be contained in the right half-plane, but there is no rigorous proof. 
\begin{figure}[!t]
\begin{minipage}{0.5\hsize}
\begin{center}
\includegraphics[width=50mm,clip]{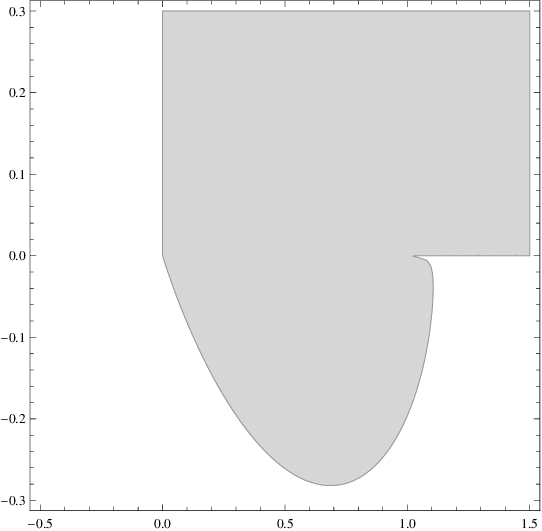}
\hangcaption{The domain $G_{B X^{3/2}}^{-1}(\C^-\cap \I\C^-)$ when $X \sim \Beta(1/2, 1.501)$}\label{dom 1}
\end{center}
\end{minipage}
\begin{minipage}{0.5\hsize}
\begin{center}
\includegraphics[width=50mm,clip]{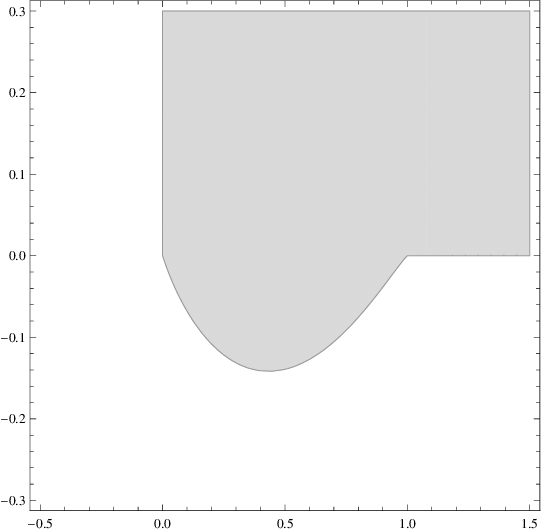}
\hangcaption{The domain $G_{B X^{3/2}}^{-1}(\C^-\cap \I\C^-)$ when $X \sim \Beta(1/2, 2)$}\label{dom 2}
\end{center}
  \end{minipage}
\begin{minipage}{0.5\hsize}
\begin{center}
\includegraphics[width=50mm,clip]{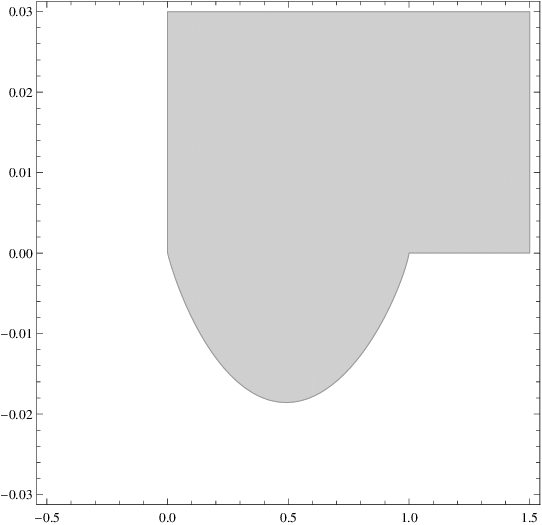}
\hangcaption{The domain $G_{B X^{4}}^{-1}(\C^-\cap \I\C^-)$ when $X \sim \Beta(1/3, 1.8)$}\label{dom 3}
\end{center}
  \end{minipage}
\begin{minipage}{0.5\hsize}
\begin{center}
\includegraphics[width=50mm,clip]{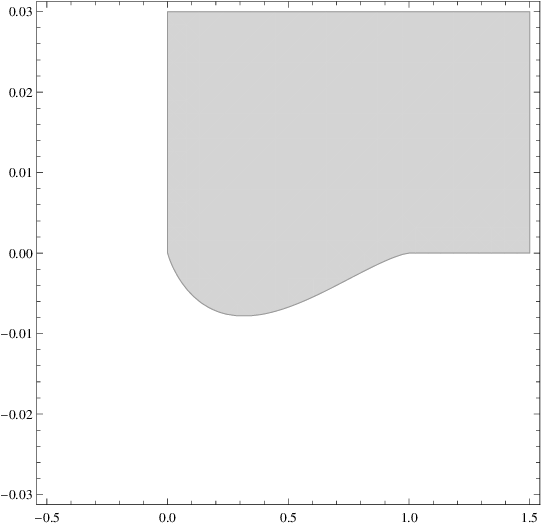}
\hangcaption{The domain $G_{B X^{4}}^{-1}(\C^-\cap \I\C^-)$ when $X \sim \Beta(1/3, 2.4)$}\label{dom 4}
\end{center}
\end{minipage}
\end{figure}

\section{Analogy between classical and free probabilities}\label{sec4}
The free analogue of the normal distribution is the semicircle distribution, and nothing else has been proposed so far. However, three different kinds of ``free gamma distributions'' have been proposed in the literature: Anshelevich defined a free gamma distribution in terms of orthogonal polynomials \cite[p.\ 238]{Ans03}; P\'erez-Abreu and Sakuma defined a free gamma distribution in terms of the Bercovici-Pata bijection \cite{PAS08}, whose property was investigated in details by Haagerup and Thorbj{\o}rnsen \cite{HT14}.
The third definition of a free gamma distribution is just the free Poisson distribution. This comes from an analogy between a characterization of gamma distributions by Lukacs and that of free Poisson distributions by Szpojankowski. 
In \cite{Luk55} Lukacs proved that non-degenerate and independent rvs $X,Y>0$ have gamma distributions with the same scale parameter if and only if $\frac{X}{X+Y}, X+Y$ are independent. The free probabilistic analogue gives us free Poisson distributions. Namely, in \cite[Theorem 4.2]{Szp1} and \cite[Theorem 3.1]{Szp2} Szpojankowski proved that bounded non-degenerate free rvs $X,Y\geq0$ such that $X\geq \epsilon I$ for some $\epsilon>0$ have free Poisson distributions with the same scale parameter if and only if $(X+Y)^{-1/2}X(X+Y)^{-1/2}, X+Y$ are free. Note that the assumption $X \geq \epsilon I$ implies that $p>1$ when $X\sim\MP(p,\theta)$. 

From the point of view of the third definition of a free gamma distribution as the free Poisson distribution, our Theorem \ref{thm3} on powers of free Poisson rvs has a good correspondence with the classical case:  

\begin{thm}\label{Power gamma}
Suppose that $Y$ follows a gamma distribution. Then $Y^r \sim\GGC$ for $r\in(-\infty,0]\cup[1,\infty)$. If $r\in(0,1)$ then $Y^r\notin\ID$.  
\end{thm}
The result for $|r|\geq1$ was first proved by Thorin \cite{Tho78}. It also follows from some facts on the class HCM: all gamma distributions are contained in $\HCM$; $\HCM \subset \GGC$; $X\sim \HCM$ implies $X^r \sim \HCM$ for $|r|\geq1$ \cite{Bon92}. The result for $r\in(-1,0)$ was recently proved by Bosch and Simon \cite{BS}. If $r\in(0,1)$ then  $Y^r$ is not ID since the tail decreases in a more rapid way than what is allowed for non-normal ID distributions (see \cite{R70} or \cite[Chap.\ 4., Theorem 9.8]{SVH03} for the rigorous statement).  

Therefore we have a complete correspondence between the ID property of powers of gamma rvs and the FID property of free Poisson rvs, except $r\in (0,1)$. We pose a conjecture on this missing interval. 
\begin{conj} \label{conj fp} 
If $X\sim \MP(p,1)$, $r\in (0,1)$ and $p>0$ then $X^r\not\sim \FID$. 
\end{conj}
A partial result can be obtained from \cite[Theorem 5.1]{Has14}: $X^r \not\sim\FID$ if $X\sim\MP(1,1)$ and $r\in (\frac{1}{2}-\frac{1}{4 n+2}, \frac{1}{2}-\frac{1}{4 n+4}) \cup (\frac{1}{2}+\frac{1}{4 n}, \frac{1}{2}+\frac{1}{4 n-2})$, $n\in\N$. Another way to check the free infinite divisibility is the Hankel determinants of free cumulants. Haagerup and M\"oller computed the Mellin transform of $X\sim \MP(1,1)$ in \cite[Theorem 3]{HM13} and the moments of $X^r$ are given by 
\be\label{Moments}
\int_0^{4}x^{r n}\, \MP(1,1)(dx)=  \frac{\Gamma(1+2 r n)}{\Gamma(1+ r n)\Gamma(2+r n)},\qquad n\in\N\cup\{0\},    
\ee
so we can compute free cumulants from these moments. According to Mathematica, the 2nd Hankel determinant $K_4 K_2-K_3^2$ of the free cumulants $\{K_n\}_{n\geq2}$ is negative for $0.35 \lesssim r<1$ and $K_6$ is negative for $0.335 \lesssim r \lesssim 0.42$ (see Figs.\ \ref{fig 50}, \ref{fig 60}). So it seems that $X^r \not\sim \FID$ for $0.335 \lesssim r <1$, but this is not a mathematical proof.   
\begin{figure}[htpb]
 \begin{minipage}{0.5\hsize}
\begin{center}
\includegraphics[width=70mm,clip]{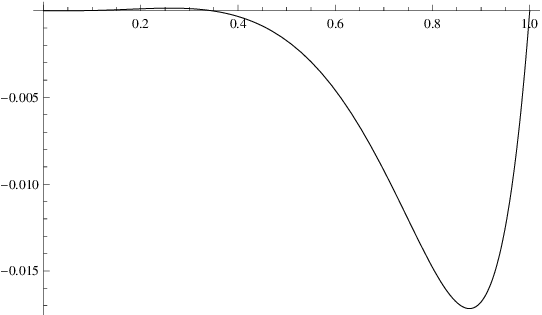}
\hangcaption{$K_4 K_2-K_3^2$ as a function of $r$.}\label{fig 50}
\end{center}
\end{minipage}
\begin{minipage}{0.5\hsize}
\begin{center}
\includegraphics[width=70mm,clip]{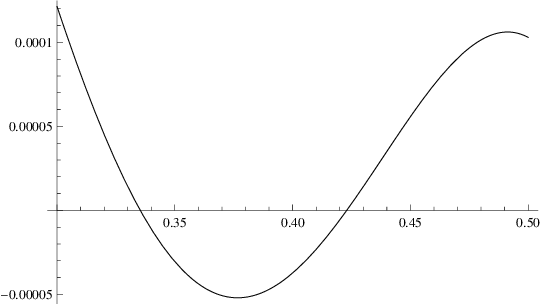}
\caption{$K_6$ as a function of $r$.}\label{fig 60}
\end{center}
\end{minipage}
\end{figure}

Our Corollary \ref{Power WS} and Corollary \ref{Sym Power WS} on powers of semicircular elements also have a good correspondence with the classical case:  %In the classical case, we can also prove  the ID property of the symmetrization of $|Z|^r$.  
\begin{thm} Suppose $Z \sim N(0,1)$. 
\begin{enumerate}[\rm(1)]
\item\label{G1} If $r \in(-\infty,0] \cup [2,\infty)$ then $|Z|^r \sim \GGC$. 
\item\label{G2} If $r \in (0,2)$ then $|Z|^r \not\sim \ID$. 
\item\label{G3} If $r \geq2$ then $|Z|^r\,\sign(Z)\sim \ID$. 
\item\label{G4} If $1\neq r<2$ then $|Z|^r \,\sign(Z)\notin\ID$. 
\end{enumerate}
\end{thm}
\begin{proof}
\eqref{G1},\eqref{G2} are a special case of Theorem \ref{Power gamma} since $Z\sim N(0,1)$ implies that $Z^2$ follows the gamma distribution $(2\pi x)^{-1/2} e^{-x/2}1_{(0,\infty)}(x)\,dx$. 

\eqref{G3} follows from \cite[Theorem 5]{SS77}. 

\eqref{G4} When $r<1, r\neq0$ then the pdf of $|Z|^r \,\sign(Z)$ is symmetric and vanishes at 0. Such a distribution is not ID by \cite[p.\ 58]{RSS90}.  When $r=0$ the distribution is a symmetric Bernoulli and so not ID. 

When $1<r<2$, one may use \cite[Theorem 2]{R70} or \cite[Chap.\ 4, Theorem 9.8]{SVH03}. 
\end{proof}
 The following problems remain to be solved in order to get the complete analogy between classical and free cases. 
\begin{conj} Suppose that  $S\sim \WS(0,1)$. 
\begin{enumerate}[\rm(1)] 
\item $|S|^r\not\sim \FID$ for $r\in(0,2)$.   
\item $|S|^r\, \sign(S) \not\sim \FID$ for $r\in(1,2)$. 
\end{enumerate}
\end{conj}
 Note that (1) is a special case of Conjecture \ref{conj fp}. The odd moments of $|S|^r \,\sign(S)$ are 0 and the even moments are the same as those of $|S|^r$. So we can compute free cumulants $\{K'_n\}_{n\geq1}$ of $|S|^r \,\sign(S)$ from \eqref{Moments}, and computation in Mathematica suggests that $K'_2 K'_6-(K'_4)^2<0$ for $1\lesssim r \lesssim 1.8$ and  $\det((K_{2i+2j-2}')_{i,j=1}^4)<0$ for $1.68\lesssim r \lesssim 1.94$ (in Figs.\ \ref{fig 100}, \ref{fig 200}). So the conjecture (2) is likely to hold.
 \begin{figure}[htpb]
 \begin{minipage}{0.5\hsize}
\begin{center}
\includegraphics[width=70mm,clip]{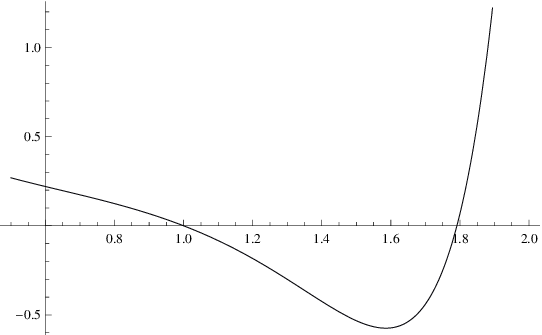}
\hangcaption{$K_2' K_6'-(K_4')^2$ as a function of $r$.}\label{fig 100}
\end{center}
\end{minipage}
\begin{minipage}{0.5\hsize}
\begin{center}
\includegraphics[width=70mm,clip]{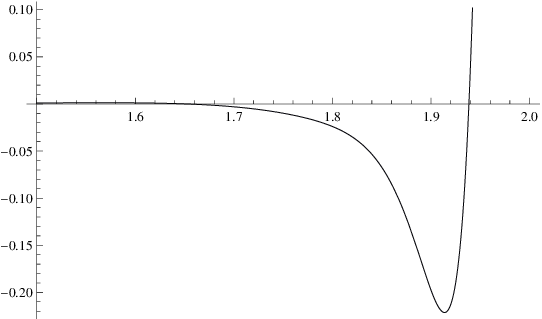}
\hangcaption{$\det((K_{2i+2j-2}')_{i,j=1}^4)$ as a function of $r$.}\label{fig 200}
\end{center}
  \end{minipage}
\end{figure}

\section*{Acknowledgements} This work is supported by JSPS Grant-in-Aid for Young Scientists (B) 15K17549. The author thanks Octavio Arizmendi and Noriyoshi Sakuma for the past collaborations and discussions which motivated this paper, Franz Lehner and Marek Bo\.zejko for discussions on powers of semicircular elements, V\'{\i}ctor P\'erez-Abreu for information on references, and Yuki Ueda and Junki Morishita for pointing out an error of Theorem \ref{thm3} of the published version.   Finally the author thanks the hospitality of organizers of International Workshop ``Algebraic and Analytic Aspects of Quantum L\'evy Processes'' held in Greifswald in March, 2015. This work is based on the author's talk in the workshop.

\begin{flushleft}
Department of Mathematics, Hokkaido University  \\
Kita 10, Nishi 8, Kitaku, Sapporo 060-0810  \\ 
Japan  

Email:thasebe@math.sci.hokudai.ac.jp

\url{http://www.math.sci.hokudai.ac.jp/~thasebe}

\end{flushleft}

\end{document}